\documentclass[letterpaper,12pt]{amsart}

\usepackage{amsmath,amsthm,amssymb,mathtools,setspace,dsfont,hyperref,color}
\renewcommand{\thispagestyle}[1]{} 
\renewcommand{\qed}{$\blacksquare$}
\newcommand{\dia}{$\diamond$}

\newtheorem{thm}{Theorem}[section]
\newtheorem{nota}[thm]{Notation}
\newtheorem{cor}[thm]{Corollary}
\newtheorem{lem}[thm]{Lemma}

\newtheorem{dfn}[thm]{Definition} 

\newtheorem{rem}[thm]{Remark}

\newcommand{\eps}{\varepsilon}
\newcommand{\vol}{\mathrm{Vol}}
\newcommand{\np}{\mathbf{NP}}
\newcommand{\C}{\mathbb{C}}
\newcommand{\N}{\mathbb{N}}
\newcommand{\Pro}{\mathbb{P}}
\newcommand{\R}{\mathbb{R}}
\newcommand{\Z}{\mathbb{Z}}
\newcommand{\bO}{\mathbf{O}}

\newcommand{\cL}{\mathcal{L}}
\newcommand{\cM}{\mathcal{M}} 
\newcommand{\cN}{\mathcal{N}}

\newcommand{\cX}{\mathcal{X}}

\DeclarePairedDelimiter\abs{\lvert}{\rvert}
\DeclarePairedDelimiter\norm{\lVert}{\rVert}

\begin{document}

\title[Condition Number Estimates for Real Polynomial Systems I]{\mbox{}
\vspace{-1in}\\ 
Probabilistic 
Condition Number Estimates for Real Polynomial Systems I: A Broader Family of 
Distributions} 

\author{Alperen Erg\"ur}
\address{ Technische Universit\"at Berlin, Institut f\"ur Mathematik, Sekretariat MA 3-2, Straße des 17. Juni 136, 10623, Berlin, Germany }
\email{erguer@math.tu-berlin.de}
\author{Grigoris Paouris}
\address{Department of Mathematics,
Texas A\&M University TAMU 3368,
College Station, Texas \ 77843-3368, USA.}
\email{grigoris@math.tamu.edu}
\author{J.\ Maurice Rojas}
\address{Department of Mathematics,
Texas A\&M University TAMU 3368,
College Station, Texas \ 77843-3368, USA.}
\email{rojas@math.tamu.edu}
\thanks{A.E.\ was partially supported by NSF grant CCF-1409020 and NSF CAREER
grant DMS-1151711. G.P.\ was partially supported by BSF grant 2010288 and NSF 
CAREER grant DMS-1151711. J.M.R.\ was partially supported by NSF grant 
CCF-1409020.}

\begin{abstract} 
We consider the sensitivity of real roots of polynomial systems 
with respect to perturbations of the coefficients. In particular ---  
for a version of the condition number defined 
by Cucker and used later by Cucker, Krick, Malajovich, and Wschebor --- we 
establish new probabilistic 
estimates that allow a much broader family of measures than considered earlier. 
We also generalize further by allowing over-determined systems. 

In Part II, we study smoothed complexity and how sparsity 
(in the sense of restricting which 
monomial terms can appear) can help further improve earlier condition number 
estimates. 
\end{abstract}  

\maketitle 

\vspace{-.3in} 
\section{Introduction}
When designing algorithms for polynomial system solving, it quickly 
becomes clear that complexity is governed by more than simply the number 
of variables and degrees of the equations. Numerical solutions are 
meaningless without further information on the spacing of the roots, 
not to mention their sensitivity to perturbation. A mathematically 
elegant means of capturing this sensitivity is the notion of 
{\em condition number} (see, e.g., \cite{bcss,cond} and our discussion  
below). 

A subtlety behind complexity bounds incorporating the condition number 
is that\linebreak computing the condition number, even within a large 
multiplicative error, is provably as hard as computing the numerical 
solution one seeks in the first place (see, e.g., \cite{demmel} for a precise 
statement in the linear case). However, it is now known 
that the condition number admits {\em probabilistic} bounds, thus 
enabling its use in average-case analysis, high probability analysis, 
and smoothed analysis of the complexity of numerical algorithms. In fact, this 
probabilistic 
approach has revealed (see, e.g., \cite{beltran,derand,lairez}) that, in 
certain settings, numerical solving can be done in polynomial-time on 
average, even though numerical solving has exponential worst-case   
complexity. 

The numerical approximation of complex roots provides an instructive example 
of how one can profit from randomization. 

First, there are classical 
reductions showing that deciding the existence of complex roots for 
systems of polynomials in $\bigcup_{m,n\in\N} (\Z[x_1,\ldots,x_n])^m$ 
is already $\np$-hard. However, classical algebraic geometry (e.g., Bertini's 
Theorem and B\'ezout's Theorem \cite{shafa})  
tells us that, with probability $1$, the number of complex roots of a 
{\em random} system of homogeneous polynomials, $P\!:=\!(p_1,\ldots,p_m)\!\in\!
\C[x_1,\ldots,x_n]$ (with each $p_i$ having fixed positive degree $d_i$), 
is $0$, $\prod^n_{i=1}d_i$, or infinite, according as $m\!>\!n-1$, $m\!=\!n-1$, 
or $m\!<\!n-1$. (Any probability measure on the coefficient space, absolutely 
continuous with respect to Lebesgue measure, will do in the preceding 
statement.)  

Secondly, examples like $P\!:=\!(x_1-x^2_2,x_2-x^2_3,\ldots,
x_{n-1}-x^2_n,(2x_n-1)(3x_n-1))$, which has affine roots 
$\left(2^{-2^{n-1}}, \ldots,2^{-2^0}\right)$ and $\left(3^{-2^{n-1}},
\ldots,3^{-2^0}\right)$, reveal that the number of digits of accuracy necessary 
to distinguish the coordinates of roots of $P$ may be exponential in $n$ 
(among other parameters). However, it is now known via earlier work on 
discriminants and random polynomial systems (see, e.g., \cite[Thm.\ 5]{pardo}) 
that the number of digits needed to separate roots of 
$P$ is polynomial in $n$ {\em with high probability}, assuming the 
coefficients are rational, and the polynomial degrees and coefficient 
heights are bounded. More simply, a classical observation from 
the theory of resultants (see, e.g., \cite{cannyphd}) is that,  
for any positive continuous probability measure on the coefficients, 
$P$ having a root with Jacobian matrix possessing small determinant is a rare 
event. So, with high probability, small perturbations of a $P$ with no 
degenerate roots should still have no degenerate roots. 
More precisely, we review below a version of the condition number 
used in \cite{SS,beltran,lairez}. Recall that the {\em singular 
values} of a matrix $T\!\in\!\R^{k\times (n-1)}$ are the (nonnegative) 
square roots of the eigenvalues of $T^\top T$, where $T^\top$ denotes the 
transpose of $T$.   
\begin{dfn}
Given $n,d_1,\ldots,d_m\!\in\!\N$ and $i\!\in\!\{1,\ldots,m\}$, 
let $N_i\!:=\!\binom{n+d_i-1}{d_i}$ and, 
for any homogenous polynomial $p_i\!\in\!\R[x_1,\ldots,x_n]$ with 
$\deg p_i\!=\!d_i$, note that the number of monomial terms of $p_i$ 
is at most $N_i$. Letting $\alpha\!:=\!(\alpha_1,\ldots,\alpha_n)$ and 
$x^\alpha\!:=\!x^{\alpha_1}_1\cdots x^{\alpha_n}_n$, let $c_{i,\alpha}$ 
denote the coefficient of $x^\alpha$ in $p_i$, and set $P\!:=\!(p_1,
\ldots,p_m)$. Also, let us define the {\em Weyl-Bombieri norms} of $p_i$ and 
$P$ to be, respectively,\\ 
\mbox{}\hfill  
$\|p_i\|_W\!:=\!\sqrt{\sum\limits_{\alpha_1+\cdots+\alpha_n=d_i} 
\frac{|c_{i,\alpha}|^2}{\binom{d_i}{\alpha}}}$ and 
$\|P\|_W\!:=\!\sqrt{\sum\limits^m_{i=1} \|p_i\|^2_W}$.\hfill \mbox{}\\   
Let $\Delta_m\!\in\!\R^{m\times m}$ be the diagonal matrix 
with diagonal entries $\sqrt{d_1},\ldots,\sqrt{d_m}$ and let\linebreak  
$DP(x)|_{T_x S^{n-1}} : T_x S^{n-1} \longrightarrow \R^m$ 
denote the linear map between tangent spaces induced by the 
Jacobian matrix of the polynomial system $P$ evaluated at the point $x$. 
Finally, when $m\!=\!n-1$, we define the {\em (normalized) 
local condition number (for solving $P\!=\!\bO$)}  
to be $\tilde{\mu}_\mathrm{norm}(P,x):=\norm{P}_{W}\sigma_\mathrm{max}
\!\left(DP(x)|^{-1}_{T_x S^{n-1}}\Delta_{n-1}\right)$ or 
$\tilde{\mu}_\mathrm{norm}
(P,x)\!:=\!\infty$, according as $DP(x)|_{T_x S^{n-1}}$ is invertible or not, 
where $\sigma_{\mathrm{max}}(A)$ is the largest singular value of a matrix 
$A$. \dia  
\end{dfn} 

\noindent 
Clearly, $\tilde{\mu}_\mathrm{norm}(P,x)\rightarrow \infty$ as $P$ approaches 
a system possessing a degenerate root $\zeta\!\in\!\Pro^{n-1}_\C$ and $x$ 
approaches $\zeta$. The intermediate normalizations in the definition are 
useful for geometric interpretations of $\tilde{\mu}_\mathrm{norm}$: There 
is in fact a natural metric $\|\cdot\|_W$ 
(reviewed in Section \ref{sec:back} and 
Theorem \ref{thm:dist} below), on the 
space of coefficients, yielding a simple and elegant algebraic relation 
between $\|P\|_W$, $\sup_{x \in S^{n-1}} \tilde{\mu}_\mathrm{norm}(P,x)$, 
and the distance of $P$ to a certain discriminant variety (see also 
\cite{M2}). But even more importantly, the preceding condition 
number (in the special case $m\!=\!n-1$) was a central ingredient in the recent 
positive solution to {\em Smale's 17th 
Problem} \cite{beltran,lairez}: For the problem of numerically approximating a 
{\em single} complex root of a polynomial system, a particular randomization 
model (independent complex Gaussian coefficients with specially chosen 
variances) enables {\em polynomial-time} average-case complexity, in the face  
of exponential deterministic complexity.\footnote{Here, ``complexity''  
simply means the total number of field operations over $\C$ needed to 
find a start point $x_0$ for Newton iteration,  
such that the sequence of Newton iterates $(x_n)_{n\in\N}$ converges to a 
true root $\zeta$ of $P$ (see, e.g., 
\cite[Ch.\ 8]{bcss}) at the rate of 
$|x_n-\zeta|\!\leq\!(1/2)^{2^{n-1}}|x_0-\zeta|$ or faster. } 

\subsection{From Complex Roots to Real Roots} 
It is natural to seek similar average-case speed-ups for the harder problem of 
numerically 
approximating real roots of real polynomial systems. However, an important 
subtlety one must consider is that the number of real roots of $n-1$ 
homogeneous polynomials in $n$ variables (of fixed degree) 
is no longer constant with probability $1$, even if the probability measure 
for the coefficients is continuous and positive. Also, small perturbations 
can make the number of real roots of a polynomial system go from 
positive to zero or even infinity. 
A condition number for real solving that takes all these subtleties into 
account was developed in \cite{cucker} and applied in the seminal series of 
papers \cite{M1,M2,M3}. In these papers, the authors performed a probabilistic 
analysis assuming the coefficients were independent real Gaussians with 
mean $0$ and very specially chosen variances.  
\begin{dfn} \cite{cucker} 
Let $\tilde{\kappa}(P,x):= \frac{\norm{P}_{W}}{\sqrt{\|P\|^2_W 
\tilde{\mu}_\mathrm{norm}(P,x)^{-2} 
+\norm{P(x)}_2^2}}$ and 
$\tilde{\kappa}(P):=\sup\limits_{x \in S^{n-1}} \tilde{\kappa}(P,x)$. 
We respectively call $\tilde{\kappa}(P,x)$ and $\tilde{\kappa}(P)$ the 
{\em local} and {\em global} condition numbers for real solving. \dia 
\end{dfn} 

\noindent 
Note that a large condition number for real solving can be caused not only 
by a root with small Jacobian determinant, but also by the existence of a 
critical point for $P$ with small corresponding critical value. So a large 
$\tilde{\kappa}$ is meant to detect the spontaneous 
creation of real roots, as well as the bifurcation of a single degenerate root 
into multiple distinct real roots, arising from small perturbations of 
the coefficients.  

Our main results, Theorems \ref{gcondition} and \ref{expectation} in 
Section \ref{sec:main1} below, show that useful condition number estimates can  
be derived for a much broader class of probability measures than considered 
earlier: Our theorems allow non-Gaussian distributions,  
dependence between certain coefficients, and, unlike the existing 
literature, our methods do not use any additional 
algebraic structure, e.g., invariance under the unitary group acting linearly 
on the variables (as in \cite{SS,M1,M2,M3}). This aspect also allows us to 
begin to address sparse polynomials (in the sequel to this paper), where 
linear changes of variables would destroy sparsity. Even better, our framework 
allows over-determined systems. 

To compare our results with earlier estimates, 
let us first recall a central estimate from \cite{M3}.  
\begin{thm} \label{CKMW} \cite[Thm.\ 1.2]{M3} 
Let $P:=(p_1,\ldots,p_{n-1})$ be a random system of homogenous $n$-variate 
polynomials where $n\!\geq\!3$ and 
$p_i(x):=\sum\limits_{\alpha_1+\cdots+\alpha_n=d_i} 
\sqrt{\binom{d_i}{\alpha}} c_{i,\alpha} x^{\alpha}$ where  
the $c_{i,\alpha}$ are independent real 
Gaussian random variables having mean $0$ and variance $1$. Then, 
letting $N:=\sum^{n-1}_{i=1}\binom{n+d_i-1}{d_i}$, $d:=\max_i d_i$, 
$M':=1+8d^2\sqrt{(n-1)^5 N \prod^{n-1}_{i=1}d_i}$, and $t\!\geq\!\sqrt{
\frac{n-1}{4\prod^{n-1}_{i=1}d_i}}$, we have:\\  
\mbox{}\hspace{1in}1. $\mathrm{Prob}(\tilde{\kappa}(P) \geq t M') \leq \frac{
\sqrt{1+\log(tM')}}{t}$\\ 
\mbox{}\hspace{1in}2. $\mathbb{E}(\log(\tilde{\kappa}(P))) \leq \log(M') + 
\sqrt{\log M'}+\frac{1}{\sqrt{\log M'}}$. 
\end{thm} 

The expanded class of distributions we allow for the coefficients of 
$P$ satisfy the following more flexible hypotheses:  
\begin{nota} 
For any $d_1,\ldots,d_m\!\in\!\N$ and $i\!\in\!\{1,\ldots,m\}$, let 
$d\!:=\!\max_i d_i$, $N_i:=\binom{n+d_i-1}{d_i}$, and  
assume $C_i\!=\!(c_{i,\alpha})_{\alpha_1+\cdots+\alpha_n=d_i}$ are 
independent random vectors in $\R^{N_i}$ with probability distributions  
satisfying:\\
\mbox{}\hspace{1cm}1.\ (Centering) For any  
$\theta \in S^{N_{i}-1}$ we have 
$ \mathbb{E}\langle C_i, \theta \rangle = 0$.\\ 
\mbox{}\hspace{1cm}2.\ (Sub-Gaussian) There is a $K>0$ such that for every 
$\theta \in S^{N_{i}-1}$ we have \\ 
\mbox{}\hspace{4.5cm}$\mathrm{Prob} 
\left( \abs{ \langle C_{i}, \theta 
\rangle } \geq t \right) \leq 2 e^{-t^2/K^2}$ for all $ t>0$.\\
\mbox{}\hspace{1cm}3.\ (Small Ball) There is a $c_0>0$ such that for every 
vector $a \in \mathbb{R}^{N_{i}}$ we have\\ 
\mbox{}\hspace{3.8cm}$\mathrm{Prob}\left( \abs{ \langle a, C_i \rangle } \leq 
\eps \norm{a}_2 \right) \leq c_0 \eps$ for all $ \varepsilon>0$. \dia 
\end{nota} 

\noindent 
By the vectors $C_i$ being independent we simply mean that 
the probability density function for the longer vector 
$C_1\times \cdots \times C_m$ can be expressed as a product of the 
form $\prod^m_{i=1}f_i(\ldots,c_{i,\alpha},\ldots)$. This is a much weaker 
assumption than having {\em all} the 
$c_{i,\alpha}$ be independent, as is usually done in the literature 
on random polynomial systems. 



A simple example of $(C_1,\ldots,C_m)$ satisfying the $3$ assumptions 
above would be to simply use an independent mean $0$ Gaussian for each 
$c_{i,\alpha}$, with each variance arbitrary. This already  
generalizes the setting of \cite{SS,M1,M2,M3} where the variances were 
specially chosen functions of $(d_1,\ldots,d_{n-1})$ and $\alpha$. 

Another example of a collection of random vectors satisfying the 
$3$ assumptions above can be obtained by letting $p\!>\!2$ and letting 
$C_i$ have the uniform distribution on\linebreak 
$B_{p}^{N_{i}}:=\left\{ x \in \R^{N_{i}} \; | \; 
\sum_{j=1}^{N_{i}} x_{j}^{p} \leq 1 \right\}$ for all $i$:  In this case the 
Sub-Gaussian assumption follows from \cite[Sec.\ 6]{BK} and the Small Ball 
Assumption is a direct consequence of the fact that $B_{p}^{N_{i}}$ 
satisfies Bourgain's Hyperplane Conjecture (see, e.g., \cite{kp}). 
Yet another important example (easier to verify) is to let the $C_i$ 
have the uniform distribution on $\ell_2$ unit-spheres of varying dimension.  

A simplified summary of  our main results (Theorems \ref{gcondition} and 
\ref{expectation} from Section \ref{sec:main1}), in the special case of square 
dense systems, is the following:  
\begin{cor} \label{cor:main1} There is an absolute constant $A\!>\!0$ 
with the following property.  Let $P:=(p_1,\ldots,p_{n-1})$ be a random system 
of homogenous $n$-variate polynomials where\linebreak  
$p_i(x):=\sum\limits_{\alpha_1+\cdots+\alpha_n=d_i} 
\sqrt{\binom{d_i}{\alpha}} c_{i,\alpha} x^{\alpha}$ and $C_i\!=\!(c_{i,\alpha})_{\alpha_1+\cdots+\alpha_n=d_i}$ are independent random vectors satisfying the 
Centering, Sub-Gaussian and Small Ball assumptions, with underlying constants 
$c_0$ and $K$. Then, for $n \geq 3$, $ d := \max_i d_i$, $d \geq 2$, 
$N:=\sum^{n-1}_{i=1}\binom{n+d_i-1}{d_i}$, and 
$M:=  A \sqrt{N} (K c_0)^{2(n-1)} (3d^2\log(ed))^{2n-3} \sqrt{n}$   
the following bounds hold:\\ 
\mbox{}\hspace{1in}1. $\mathrm{Prob}(\tilde{\kappa}(P) \geq t M) \leq 
\left\{ \begin{matrix}\mbox{} \ \ \ \ \ \ \ \ \ \ \ 3 t^{-\frac{1}{2}}\ \ \ \ \ \ \ 
\ \ \ \ \ \ \ ; \text{ if } 1 \leq t \leq (ed)^{2(n-1)}\\
3t^{-\frac{1}{2}} \left( \frac{t}{(ed)^{2(n-1)}} 
\right)^{\frac{1}{4\log(ed)}} ; 
\text{ if } t \geq (ed)^{2(n-1)}\end{matrix}\right. $\\ 
\mbox{}\hspace{1in}2. $\mathbb{E}(\log\tilde{\kappa}(P)) \leq 1+\log M$. 
\end{cor}

\noindent 
Corollary \ref{cor:main1} is proved in Section \ref{sec:main1}.  Theorems 
\ref{gcondition} and \ref{expectation} in Section \ref{sec:main1} below in 
fact state much stronger estimates than our simplified summary above. 

Note that, for fixed $d$ and $n$, the bound from Assertion (1) of 
Corollary \ref{cor:main1} shows a somewhat slower rate of decay for the 
probability of a large condition number than the older bound from Assertion 
(1) of Theorem \ref{CKMW}: $O(1/t^{0.3523})$ vs.\ $O(\sqrt{\log t}/t)$. 
However, the older $O(\sqrt{\log t}/t)$ bound was restricted to a special 
family of Gaussian distributions (satisfying invariance with 
respect to a natural $O(n)$-action on the root space $\Pro^{n-1}_\R$) 
and assumes $m\!=\!n-1$. Our techniques come 
from geometric functional analysis, work for a broader family of 
distributions, and we make no group-invariance assumptions.  

Furthermore, our techniques allow condition number bounds in a new setting: 
over-determined systems, i.e., $m \times n$ 
systems with $m\!>\!n-1$. See the next section for the definition 
of a condition number enabling 
$m\!>n-1$, and the statements of Theorems \ref{gcondition} and 
\ref{expectation} for our most general condition number bounds. 
The over-determined case occurs in many important applications involving 
large data, where one may make multiple redundant measurements of some physical 
phenomenon, e.g., image reconstruction from multiple projections. 
There appear to have been no probabilistic condition number estimates for the 
case $m\!>\!n-1$ until now. In particular, for $m$ proportional to $n$, 
we will see at the end of this paper how our condition number estimates are 
close to optimal.  

To the best of our knowledge, the only 
other result toward estimating condition numbers of 
non-Gaussian random polynomial systems is due to Nguyen \cite{Ng}. However,
in \cite{Ng} the degrees of the polynomials are assumed to be bounded by a 
small fraction of the number of variables, $m\!=\!n-1$, and the quantity
analyzed in \cite{Ng} is not the condition number considered in
\cite{SS} or \cite{M1,M2,M3}.

The precise asymptotics of the decay rate for the 
probability of having a large condition number remain unknown, 
even in the restricted Gaussian case considered by Cucker, Malajovich, Krick, 
and Wschebor. So we also prove {\em lower} bounds 
for the condition number of a random polynomial system.
To establish these bounds, we will need one more assumption on the randomness.

\begin{nota} 
For any $d_1,\ldots,d_m\!\in\!\N$ and $i\!\in\!\{1,\ldots,m\}$, let
$d\!:=\!\max_i d_i$, $N_i:=\binom{n+d_i-1}{d_i}$, and
assume $C_i\!=\!(c_{i,\alpha})_{\alpha_1+\cdots+\alpha_n=d_i}$ is an
independent random vector in $\R^{N_i}$ with probability distribution
satisfying:\\
4. (Euclidean Small Ball) There is a constant $\tilde{c}_0>0$ such that for 
every $\eps\!>\!0$ we have\\  
\mbox{}\hspace{2.6cm} 
$\mathrm{Prob}\left( \| C_i\|_{2} \leq \varepsilon \sqrt{N_i} \right) 
\leq ( \tilde{c_0} \varepsilon)^{N_i}$. \dia 
\end{nota} 
\begin{rem} 
If the vectors $C_i$ have independent coordinates satisfying the Centering 
and Small Ball Assumptions, then Lemma
\ref{smallBallTensorize} from Section \ref{sub:smallball} implies  
that the Euclidean Small Ball Assumption holds as well. Moreover, if the $C_i$
are each uniformly distributed on a convex body $X$ and satisfy our Centering 
and Sub-Gaussian assumptions, then a result of Jean Bourgain \cite{Bou} (see 
also \cite{DP} or \cite{KM} for alternative proofs) implies that both the Small 
Ball and Euclidean Small Ball Assumptions hold, and with $\tilde{c}_0$ depending only on 
the Sub-Gaussian constant $K$ (not the convex body $X$). \dia 
\end{rem} 
\begin{cor} 
\label{getlog}
Suppose $n,d\geq 3$, $m=n-1$, and $d_j=d$ for all $j\!\in\!\{1,\ldots, 
n-1\}$. Also let $P:= (p_1, \ldots , p_m)$ be a random polynomial system 
satisfying our Centering, Sub-Gaussian, Small Ball, and Euclidean Small Ball assumptions, 
with respective underlying constants $K$ and $\tilde{c_0}$. 
Then there are constants $A_2\!\geq\!A_1\!>\!0$ such that  
$$ A_1 (n\log(d) + d\log(n))  \leq \mathbb{E}(\log \tilde{\kappa} (P)) \leq A_2 
(n\log(d)+ d\log(n)). \ \ \ \text{\qed}$$ 
\end{cor}  

\noindent 
Corollary \ref{getlog} follows immediately from a more general
estimate: Lemma \ref{lower-bound} from Section \ref{sub:smallball}.
It would certainly be more desirable to know bounds within a constant 
multiple of $\tilde{\kappa}(P)$ instead. We 
discuss more refined estimates of the latter kind in Section \ref{sub:lower}, 
after the proof of Lemma \ref{lower-bound}. 

As we close our introduction, we point out that one of the tools we developed 
to prove our main theorems may be of independent interest: Theorem 
\ref{Kellogsys} of the next section extends, to polynomial systems,  
an earlier estimate of Kellog \cite{kellog} on the norm of the derivative 
of a single multivariate polynomial. 

\section{Technical Background} 
\label{sec:back} 
We start by defining an inner product structure on spaces of polynomial 
systems. For $n$-variate degree $d$ homogenous polynomials 
$f(x):=\sum_{|\alpha|=d} b_{\alpha}x^{\alpha}, 
g(x):=\sum_{|\alpha|=d} c_{\alpha}x^{\alpha}\in\R[x_1,\ldots,x_n]$, 
their {\em Weyl-Bombieri inner product} is defined as 
$$ \langle f,g \rangle_{W} := \sum_{|\alpha|=d} \frac{b_{\alpha}c_{\alpha}}
{\binom{d}{\alpha}}. $$

It is known (see, e.g., \cite[Thm.\ 4.1]{kostlan}) 
that for any $U \in O(n)$ we have 
$$ \langle f \circ U , g \circ U \rangle_{W} = \langle f ,g \rangle_{W}. $$ 
Let $D:=(d_1,\ldots , d_m)$ and let $H_D$ denote the space of (real) $m\times 
n$ systems of homogenous $n$-variate polynomials with respective degrees  
$d_1, \ldots , d_m$. Then for $F:=(f_1,\ldots,f_m) \in H_{D}$ and 
$G:=(g_1, \ldots, g_m) \in H_{D}$ we define the Weyl-Bombieri inner product 
{\em for two polynomial systems} to be 
$ \langle F ,G \rangle_{W}: = \sum_{i=1}^m \langle f_i , g_i \rangle_{W}$. 
We also let $\|F\|_W\!:=\!\sqrt{\langle F,F\rangle}$. 

A geometric justification for the definition of the condition number 
$\tilde{\kappa}$ can then be derived as follows: First, for $x \in S^{n-1}$, we 
abuse notation slightly by also letting $DP(x)$ denote the $m\times n$ Jacobian 
matrix of $P$, evaluated at the point $x$. For $m=n-1$ we denote 
the set of polynomial systems with singularity at $x$ by \\ 
\mbox{}\hfill 
$\Sigma_{\R}(x):=\{ P \in H_{D} \; | \; x \; \text{is a multiple root of} \; P \} $\hfill \mbox{}\\ 
and we then define $\Sigma_{\R}$ (the {\em real part of the disciminant 
variety for $H_D$}) to be: \\ 
\mbox{} \hfill $\Sigma_{\R}:=\{ P \in H_{D} \; | P \; \text{has a multiple root in} \; S^{n-1} \}= \bigcup_{x \in S^{n-1}} \Sigma_{\R}(x)$.\hfill 
\mbox{}\\  
Using the Weyl-Bombieri inner-product to define the underlying distance, 
we point out the following important geometric characterization of 
$\tilde{\kappa}$: 
\begin{thm} \label{thm:dist} \cite[Prop.\ 3.1]{M2} When $m\!=\!n-1$ we have 
$\tilde{\kappa}(P) = \frac{\norm{P}_{W}}{\mathrm{Dist}(P,\Sigma_{\R})}$ 
for all $P\!\in\!H_D$. \qed 
\end{thm}

We call a polynomial system $P\!=\!(p_1,\ldots,p_m)$ with $m\!=\!n-1$ 
(resp.\ $m\!\geq\!n$) {\em square} (resp.\ {\em over-determined}). 
Newton's method for over-determined systems was studied in \cite{DS1}. 
So now that we have a geometric characterization of the condition number for 
square systems it will be useful to also have one for over-determined systems. 
\begin{dfn} 
\label{dfn:newcond} 
Let $\sigma_\mathrm{min}(A)$ denote the smallest singular value of a matrix 
$A$. For any system of homogeneous polynomials 
$P\!\in\!(\R[x_1,\ldots,x_n])^m$ set\\  
\mbox{}\hfill $L(P,x) := \sqrt{\sigma_\mathrm{min}
\left(\Delta^{-1}_m DP(x)|_{T_x S^{n-1}}\right)^2 
+ \norm{P(x)}_2^2}$,\hfill\mbox{}\\    
We then define 
$\tilde{\kappa}(P,x)=\frac{\norm{P}_W}{L(P,x)}$ 
and 
$\tilde{\kappa}(P)=\sup\limits_{x\in S^{n-1}}\kappa(P,x)$. \dia  
\end{dfn}  

\noindent 
The quantity $\min\limits_{x \in S^{n-1}}L(P,x)$ thus plays 
the role of $\mathrm{Dist}(P,\Sigma_{\R})$ in the more general 
setting of $m\!\geq\!n-1$. We now recall an important observation from 
\cite[Sec.\ 2]{M2}:\linebreak 
Setting $D_x(P):=DP(x)|_{T_x S^{n-1}}$ we have 
$\sigma_{\mathrm{min}}(\Delta^{-1}_{n-1}D_x(P))= \sigma_\mathrm{max}
\left(D_x(P)^{-1}\Delta_{n-1}\right)^{-1}$, 
when $m\!=\!n-1$ and $D_x(P)$ is invertible. So by the 
definition of $\tilde{\mu}_\mathrm{norm}(P,x)$
 we have $$ L(P,x)= \sqrt{\sigma_\mathrm{max}\left(D_x(P)^{-1} 
\Delta_{n-1}\right)^{-2} + \norm{P(x)}_2^2}=
\sqrt{\|P\|_W^2 \tilde{\mu}_\mathrm{norm}(P,x)^{-2} + \norm{P(x)}_2^2} $$ 
\noindent and thus our more general definition agrees with the classical 
definition in the square case. 

Since the $W$-norm of a random polynomial system has strong concentration 
properties for a broad variety of distributions (see, e.g., 
\cite{V}), we will be interested in the behavior of $L(P,x)$. So let us 
define the related quantity 
$\cL(x,y):= \sqrt{\norm{\Delta^{-1}_m D^{(1)}P(x)(y)}_2^2+
\norm{P(x)}_2^2}$. It follows directly that 
$L(P,x)=\inf\limits_{\substack{y \bot x \\ y\in S^{n-1}}} \cL(x,y)$. 
  	
We now recall a classical result of O.\ D.\ Kellog. The theorem below is a 
summary of \cite[Thms.\ 4--6]{kellog}.  
\begin{thm} \cite{kellog} \label{kellog} 
Let $p\!\in\!\R[x_1,\ldots,x_n]$ have degree $d$ and set 
$\norm{p}_{\infty}:=\sup_{x \in S^{n-1}}\abs{p(x)}$ and 
$\norm{D^{(1)}p}_{\infty}:=\max_{x,u \in S^{n-1}} \abs{D^{(1)}p(x)(u)}$.   
Then: 
\begin{enumerate}
\item We have $\norm{D^{(1)}p}_{\infty} \leq  d^2 \norm{p}_{\infty}$ and, 
for any mutually orthogonal $x,y\!\in\!S^{n-1}$, we also have 
$\abs{D^{(1)}p(x)(y)}  \leq d \norm{p}_{\infty}$. 

\item If p is homogenous then we also have 
$\norm{D^{(1)}p}_{\infty} \leq  d \norm{p}_{\infty}$. \qed 
\end{enumerate}
\end{thm}

For any system of homogeneous polynomials 
$P:=(p_1,\ldots , p_m)\!\in\!(\R[x_1,\ldots,x_n])^m$ 
define $ \norm{P}_{\infty}:= \sup_{x \in S^{n-1}} 
\sqrt{\sum_{i=1}^m p_i(x)^2}$. 
Let $DP(x)(u)$ denote the image of the vector $u$ under the 
linear operator $DP(x)$, and set 
$$ \left\|D^{(1)}P\right\|_{\infty}:= \sup_{x,u\in S^{n-1}} \norm{D P(x)(u)}_2 
= \sup_{x,u \in S^{n-1}} \sqrt{\sum_{i=1}^m \langle \nabla p_i(x) , u 
\rangle^2}. $$ 
\begin{thm} \label{Kellogsys}
Let $P:=(p_1,\ldots,p_m)\!\in\!(\R[x_1,\ldots,x_n])^m$ be a polynomial system 
with $p_i$ homogeneous of degree $d_i$ for each $i$ and set 
$d\!:=\!\max_i d_i$. Then:  
\begin{enumerate}
\item We have $\norm{D^{(1)}P}_{\infty} \leq d^2 \norm{P}_{\infty}$ and, 
for any mutually orthogonal $x,y\!\in\!S^{n-1}$, we also have  
$\norm{D P(x)(y)}_2  \leq d \norm{P}_{\infty}$. 
\item If $\deg(p_i)=d$ for all $i\in\{1,\ldots,m \}$ then we also 
have $\norm{D^{(1)}P}_{\infty} \leq d \norm{P}_{\infty}$. 
\end{enumerate}
\end{thm}

\noindent 
{\em Proof.} Let $(x_0,u_0)$ be such that   
$\norm{D^{(1)}P}_{\infty}= \norm{D P(x_0)(u_0)}_2$ and let 
$\alpha:= (\alpha_1, \ldots , \alpha_m)$ where 
$\alpha_i:= \frac{\langle \nabla p_i(x_0) , u_0 \rangle}
{\norm{D^{(1)}P}_{\infty}}$. Note that $\norm{\alpha}_2=1$. Now define a 
polynomial $q\!\in\!\R[x_1,\ldots,x_n]$ of degree $d$ 
via $q(x):= \alpha_1 p_1(x) + \cdots + \alpha_m p_m(x)$ 
and observe that 

$$ \nabla q(x) = \left( 
\alpha_1 \frac{\partial p_1}{\partial x_1} + \cdots + \alpha_m 
\frac{\partial p_m}{\partial x_1}, 
\ldots , \alpha_1 \frac{\partial p_1}{\partial x_n}  
+ \cdots + \alpha_m \frac{\partial p_m}{\partial x_n} \right),$$
$$  \langle \nabla q , u \rangle = 
u_1 \left(\alpha_1 \frac{\partial p_1}{\partial x_1} + \cdots 
+ \alpha_m \frac{\partial p_m}{\partial x_1}
\right) + \cdots + u_n \left(\alpha_1 \frac{\partial p_1}{\partial x_n} 
+ \cdots + \alpha_m \frac{\partial p_m}{\partial x_n} \right), $$

\noindent 
and $ \langle \nabla q(x) , u \rangle  = \sum_{i=1}^m \alpha_i \langle \nabla p_i(x) , u \rangle$. In particular, for our chosen $x_0$ and $u_0$, we have  

$$ \langle \nabla q(x_0) , u_0 \rangle  = \sum_{i=1}^m \alpha_i \langle \nabla 
p_i(x_0) , u_0 \rangle =  \sum_{i=1}^m \frac{\langle \nabla p_i(x_0) , u_0 
\rangle^2}{\norm{D^{(1)}P}_{\infty}} = \left\|{D^{(1)}P}_{\infty}\right\|. $$

Using the first part of Kellog's Theorem we have

$$ \norm{D^{(1)}P}_{\infty} \leq \sup_{x,u \in S^{n-1}} \abs{\langle \nabla q(x) , u \rangle} \leq d^2 \norm{q}_{\infty}. $$

\noindent Now we observe by the Cauchy-Schwarz Inequality that 
$$ \norm{q}_{\infty}= \sup_{x \in S^{n-1}} \left|\sum_{i=1}^m \alpha_i 
p_i(x)\right| \leq \sup_{x \in S^{n-1}} \sqrt{\sum_{i=1}^m p_i(x)^2}.$$  

\noindent So we conclude that  
$ \norm{D^{(1)}P}_{\infty} \leq d^2 \norm{q}_{\infty} \leq 
d^2 \sup_{x \in S^{n-1}} \sqrt{\sum_{i=1}^m p_i(x)^2}= 
d^2 \norm{P}_{\infty}$. 
We also note that when $\deg(p_i)=d$ for all $i$, the 
polynomial $q$ is homogenous of degree $d$. So for this special case, the 
second part of Kellog's Theorem directly implies $\norm{D^{(1)}P}_{\infty} \leq 
d\norm{P}_{\infty}$. 

For the proof of the first part of Assertion (1) we define 
$\alpha_i=\frac{\langle \nabla p_i(x) , y \rangle}{\norm{D P(x)(y)}_2}$ 
and\linebreak  
$q(x)=\alpha_1 p_1 + \cdots + \alpha_n p_n$. Then 
$\langle \nabla q(x) , y \rangle= \sum_{i} \alpha_i \langle \nabla p_i(x) ,y 
\rangle = \norm{D P(x)(y)}_2$. 

\noindent By applying Kellog's Theorem on the orthogonal direction $y$ we 
then obtain  

$$ \norm{D P(x)(y)}_2 = \langle \nabla q(x) , y \rangle \leq d \norm{q}_{\infty} \leq d \norm{P}_{\infty}. \ \ \ \ \ \ \ \  \text{\qed} $$ 

\medskip
Using our extension of Kellog's Theorem to polynomial systems, we  
develop useful estimates for $\norm{P}_{\infty}$ and 
$\norm{D^{(i)}P}_{\infty}$. In what follows, we call a subset 
$\cN$ of a metric space $X$ a {\em $\delta$-net on $X$} if and only if 
the every point of $X$ is within distance $\delta$ of some point of 
$\cN$. A basic fact we'll use repeatedly is that, for any 
$\delta>0$ and {\em compact} $X$, one can always find a finite 
$\delta$-net for $X$. 
\begin{lem} \label{net-norm}
Let $P:=(p_1,\ldots,p_m)\in(\C[x_1,\ldots,x_n])^m$ be a system of homogenous 
polynomials, $\cN$ a $\delta$-net on $S^{n-1}$, and set $d:=\max_i d_i$. Let 
$\max_{\cN}(P):=\sup_{y \in \cN}  \norm{P(y)}_2$. Similarly let us 
define $\max_{\cN^{k+1}}(D^{(k)}P):=\sup_{x,u_1,\ldots,u_k \in \cN} 
\norm{D^{(k)} P(x) (u_1, \ldots, u_k)}_2$, and set \linebreak   
$\left\|D^{(k)} P\right\|_{\infty}:=\sup_{x,u_1,\ldots,u_k \in S^{n-1}} 
\left\|D^{(k)} P(x) (u_1,\ldots, u_k)\right\|_2$. Then: 
\begin{enumerate}
\item $\norm{P}_{\infty} \leq \frac{\max_{\cN}(P)}
{1-\delta d^2 } \text{ and } \norm{D^{(k)} P}_{\infty} \leq 
\frac{\max_{\cN^{k+1}}(D^{(k)}P) }{1-\delta d^2 \sqrt{k+1}}$. 
\item If $\deg(p_i)=d$ for each $ i \in \{1,\ldots,m \}$ then we have 
$$ \norm{P}_{\infty} \leq \frac{\max_{\cN}(P)}{1-\delta d} 
\text{ and } 
\norm{D^{(k)} P}_{\infty} \leq \frac{\max_{\cN^{k+1}}(D^{(k)}P)}
{1-\delta d \sqrt{k+1}}. $$
\end{enumerate}

\end{lem}

\begin{proof} We first prove Assertion (2). Observe that the Lipschitz 
constant of $P$ on $S^{n-1}$ is bounded from 
above by $\norm{D^{(1)}p}_{\infty}$: This can be seen by taking 
$x,y \in S^{n-1}$ and considering the integral 
$P(x)-P(y)=  \int_0^{1} DP(y+t(x-y))(x-y) \; dt$.  

Since $\norm{y+t\cdot (x-y)}_2 \leq 1$ for all $t \in [0,1]$, the homogeneity 
of the system $P$ implies  

$$\norm{DP(y+t(x-y))(x-y)}_2 \leq \norm{D^{(1)}P}_{\infty} \norm{x-y}_2 $$ 

Using our earlier integral formula, we conclude that  
$\norm{P(x)-P(y)}_2 \leq \norm{D^{(1)}P}_{\infty} \norm{x-y}_2$. 

Now, when the degrees of the $p_i$ are identical,
let the Lipschitz constant of $P$ be $M$. By Assertion (2) of 
Theorem \ref{Kellogsys} we have $M \leq \norm{D^{(1)}P}_{\infty} \leq d 
\norm{P}_{\infty}$. Let $x_0 \in S^{n-1}$ be such that 
$\norm{P(x_0)}_2=\norm{P}_{\infty}$ and 
let $y \in \cN$ satisfy $\abs{x_0-y} \leq \delta$. Then 
$\norm{P}_{\infty}=\norm{P(x_0)}_2 \leq \norm{P(y)}_2 + \norm{x_0-y}_2 M 
\leq \max_{\cN}(P)+ \delta d \norm{P}_{\infty}$, 
and thus
$$ (\star) \ \ \ \ \ \ \ \ \ \ \ \ \ \ \ \ \ \ \ \ 
\ \ \ \ \ \ \ \ \ \ \ \ \ \ \ \ 
\norm{P}_{\infty}(1-d \delta) \leq \max_{x\in\cN}P(x).\ \ \ \ \ \ \ \ \ \ \ \ 
\ \ \ \ \ \ \ \ \ \ \ \ \ \ \ \ \ \ \ \ \ \ \ \ \mbox{}$$ 

To bound the norm of $D^{(k)} P(x) (u_1, \ldots, u_k)$ let us consider the net 
defined by\linebreak 
$\cN \times \cdots \times \cN = \cN^{k+1}$ on $ S^{n-1} \times \cdots \times 
S^{n-1}$. Let $x:=(x_1, \ldots , x_{k+1}) \in S^{n-1} \times \cdots \times 
S^{n-1}$ and $y:=(y_1,\ldots,y_{k+1}) \in \cN^{k+1}$ be such that 
$\norm{x_i-y_i}_2 \leq \delta$ for all $i$. Clearly,  
$\norm{x-y}_2 \leq \delta \sqrt{k+1}$. Since $x$ was arbitrary, this 
argument proves that $\cN^{k+1}$ is a $\delta\sqrt{k+1}$-net. Note also that 
$D^{(k)} P(x) (u_1, \ldots, u_k)$ is a homogenous polynomial system with 
$(k+1)n$ variables and degree $d$. The desired bound then follows from 
Inequality ($\star$) obtained above. 

To prove Assertion (1) of our current lemma, the preceding proof carries over 
verbatim, simply employing Assertion (1), instead of Assertion (2), from 
Theorem \ref{Kellogsys}.  
\end{proof}

\section{Condition Number of Random Polynomial Systems}\label{CC}
\subsection{Introducing Randomness}
Now let $P:=(p_1,\ldots , p_m)$ be a {\em random} polynomial system where 
$p_j(x):=\sum_{|\alpha|=d_j} c_{j,\alpha} \sqrt{\binom{d_j}{\alpha}} 
x^{\alpha}$. In particular, recall that $N_j= { n+d_j -1 \choose d_j}$ 
and we let $C_j=\left(c_{j,\alpha}\right)_{|\alpha|=d_j}$ be a random 
vector in $\R^{N_{j}}$ satisfying the Centering, Sub-Gaussian, and Small Ball 
assumptions from the introduction. Letting 
$\cX_j:=\left(\sqrt{\binom{d_j}{\alpha}}x^{\alpha}
\right)_{|\alpha|=d_j}$ we then have $p_j(x)=\langle C_j , 
\cX_j\rangle$. 
In particular, recall that the Sub-Gaussian assumption is that 
there is a $K\!>\!0$ such that for each $\theta \in S^{N_j-1}$ and $t\!>\!0$ 
we have $\mathrm{Prob} \left( \abs{ \langle C_{j}, \theta 
\rangle } \geq t \right) \leq 2 e^{-t^2/K^2}$. 
Recall also that the Small Ball assumption is that there is a $c_0>0$ such 
that for every vector $a \in \mathbb{R}^{N_{i}}$ and $\eps\!>\!0$ we have
$\mathrm{Prob}\left( \abs{ \langle a, C_j \rangle } \leq 
\eps \norm{a}_2 \right) \leq c_0 \eps$.
In what follows, several of our bounds will depend on the  
parameters $K$ and $c_0$ underlying the random variable being Sub-Gaussian 
and having the Small Ball property. 

For any random variable $\xi$ on $\R$ we denote its median 
by $\mathrm{Med}(\xi)$. Now, if 
$\xi:=\left| \langle C_{j} , \theta \rangle \right|$, then setting 
$t\!:=\!2K$ in the Sub-Gaussian assumption for $C_j$ yields 
$\mathrm{Prob} ( \xi \geq 2K) \leq \frac{1}{2}$, i.e.,  
$\mathrm{Med}(\xi) \leq 2 K$. On the other hand, setting 
$\varepsilon:= \frac{1}{2 c_0}$ in the Small Ball assumption for 
$C_j$ yields $\mathrm{Prob} ( \xi \leq \frac{1}{ 2c_0}) \leq \frac{1}{2}$,  
i.e., $\mathrm{Med}(\xi) \geq \frac{1}{2c_0}$. Writing 
$1\!=\!\mathrm{Med}(\xi)\cdot \frac{1}{\mathrm{Med}(\xi)}$ we 
then easily obtain  
\begin{eqnarray} 
\label{ineq:kc} 
Kc_0 & \geq & \frac{1}{4}. 
\end{eqnarray} 

\noindent 
In what follows we will use Inequality (\ref{ineq:kc}) several times.

\subsection{The Sub-Gaussian Assumption and Bounds Related to Operator Norms}. 
We will need the following inequality, reminiscent of Hoeffding's classical 
inequality \cite{hoeffding}. 
\begin{thm} \cite[Prop.\ 5.10]{V} \label{Bernstein} 
There is an absolute constant $c\!>\!0$ with the following 
property: If $X_1,\ldots, X_n$ are Sub-Gaussian random variables with mean 
zero and underlying constant $K$, and $a=(a_1,\ldots,a_n) \in \mathbb{R}^n$ 
and $t \geq 0$, then 
$$  \mathrm{Prob}\left( \left|\sum_i a_i X_i\right| \geq t  \right) \leq 2 
\exp\left(\frac{-ct^2}{K^2 \norm{a}_2^2 }\right). \ \ \ \ \text{ \qed}$$ 
\end{thm}

\begin{lem}\label{operatornorm}
Let $P:=(p_1,\ldots,p_m)$ be a random polynomial system where, as before, 
$p_j(x)=\sum_{|\alpha|=d_j} c_{j,\alpha}\sqrt{\binom{d_j}{\alpha}} 
x^{\alpha}$ and the the coefficient vectors $C_j$ are independent random 
vectors satisfying 
the Centering, Sub-Gaussian, and Small Ball assumptions from the introduction, 
with underlying constants $K$ and $c_0$. 
Then, for $\cN$ a $\delta$-net over $S^{n-1}$ and $t \geq 2$, we have the 
following inequalities: 
\begin{enumerate}
\item If $\deg(p_j)=d$ for all $j \in \{ 1, \ldots, m \}$ then 
$$ \mathrm{Prob} \left( \norm{P}_{\infty} \leq \frac{ 2 t K\sqrt{m}}
{1-d\delta} \right) \geq 1-2 \abs{\cN} e^{-O(t^2 m)} $$

\noindent 
In particular, there is a constant $c_1\!\geq\!1$ such that 
for $\delta=\frac{1}{3d}$ and 
$t=s\log(ed)$ with $s \geq 1$ we have $\mathrm{Prob} \left( \norm{P}_{\infty} 
\leq  3 s K \sqrt{m} \log(ed) \right) \geq 1 -  e^{-c_1 s^2 m \log(ed)}$. 

\item If $d:=\max_j \deg p_j$ then 

$$ \mathrm{Prob} \left(\norm{P}_{\infty} \leq \frac{2 t K\sqrt{m}}
{1-d^2\delta} \right) \geq 1-2 \abs{\cN} e^{-O(t m)} $$

\noindent 
In particular, there is a constant $c_2\!\geq\!1$ such that 
for $\delta=\frac{1}{3d^2}$, $t=s\log(ed)$ with $s \geq 1$, 
we have $\mathrm{Prob} \left( \norm{P}_{\infty} \leq  3 s K \sqrt{m} \log(ed) 
\right) \geq 1 - e^{-c_2 s^2 m \log(ed)}$.  
\end{enumerate}

\end{lem}

\begin{proof}
We prove Assertion (2) since the proofs of the two assertions are virtually 
identical. First observe that the identity 
$(x_1^2 + \cdots + x_n^2)^d=\sum_{|\alpha|=d} 
\binom{d}{\alpha} x^{2\alpha} $ implies $ \norm{\cX_j}_2=1$ for all 
$j\leq m$. Using our Sub-Gaussian assumption on the random vectors $C_j$, and 
the fact that $ p_j(x)= \langle C_j , \cX_j\rangle$, we obtain that 
$ \mathrm{Prob} \left( |p_j(x)| \geq t \right) \leq 2e^{-t^2/K} $
for every $x \in S^{n-1}$.  

Now we need to tensorize the preceding inequality. By Theorem \ref{Bernstein},  
we have for all $a \in S^{m-1}$ that  
$ \mathrm{Prob}\left( \abs{\langle a, P(x) \rangle} \geq t \right) \leq 2 
e^{-ct^2/K^2}$.  Letting $\cM$ be a $\delta$-net on $S^{m-1}$ we then 
have  $\mathrm{Prob}\left( \max_{a \in \cM} \abs{\langle a, P(x) 
\rangle} \geq t  \right) \leq 2 \abs{\cM} e^{-ct^2/K^2}$,  
where we have used the classical\linebreak 
union bound for the multiple events defined 
by the (finite) $\delta$-net $\cM$. Since\linebreak 
$\norm{P(x)}_2=
\max_{\theta \in S^{m-1}} \abs{\langle \theta, P(x) \rangle}$, an application 
of Lemma \ref{net-norm} for the linear polynomial
$\langle \; \cdot \; , P(x) \rangle$ gives us $\mathrm{Prob}\left(\norm{P(x)}_2 
\geq  \frac{t \sqrt{m} K}{1-\delta} \right) \leq 2 \abs{\cM} e^{-c t^2m}$. 

It is known that for any $\delta\!>\!0$, $S^{m-1}$ admits a $\delta$-net 
$\cM$ such that 
$ \abs{\cM} \leq \left(\frac{3}{\delta}\right)^m$ (see, e.g,  
\cite[Lemma 5.2]{V}). So for $t \geq 1$ and $\delta=\frac{1}{2}$ we have
$\mathrm{Prob}\left( \norm{P(x)}_2 \geq  2 t \sqrt{m} K \right)  \leq 2 
e^{-c_2 t^2 m}$ for some suitable constant $c_2\!\geq\!c$. We have thus 
arrived at a 
point-wise estimate on $\norm{P(x)}_2$. Doing a union bound on a $\delta$-net  
$\cN$ now on $S^{n-1}$ we then obtain: 

$$ \mathrm{Prob}\left( \max_{x \in \cN} \norm{P(x)}_2 \geq 2 t \sqrt{m} K 
\right) \leq 2 \abs{\cN} e^{-c_1 t^2m}. $$

\noindent 
Using Lemma \ref{net-norm} once again completes our proof. 
\end{proof}

Theorem \ref{Kellogsys} and Lemma \ref{operatornorm} then 
directly imply the following:
\begin{cor}  Let $P$ be a random polynomial system as in Lemma 
\ref{operatornorm}. Then there are constants $c_1,c_2\!\geq\!1$ such that 
the following inequalities hold for $s \geq 1$:  
\begin{enumerate}
\item If $\deg(p_j)=d$ for all $j \in \{ 1,\ldots,m\}$ then both 
$\mathrm{Prob} \left( \norm{D^{(1)}P}_{\infty} \leq  3s K \sqrt{m} d \log(ed) 
\right)$\linebreak 
\scalebox{.95}[1]{and $\mathrm{Prob} \left( \norm{D^{(2)}P}_{\infty} \leq   
3s K \sqrt{m} d^2 \log(ed)  \right)$ are bounded from below by 
$1-2e^{-c_1 s^2 m \log(ed)}$.}  

\item If \mbox{$d:=\max_j \deg p_j$} then both 
$\mathrm{Prob} \left( \norm{D^{(1)}P}_{\infty} \leq  3s K \sqrt{m} d^2 
\log(ed) \right)$ and\linebreak 
\scalebox{.97}[1]{$\mathrm{Prob} \left( \norm{D^{(2)}P}_{\infty} \leq   
3s K \sqrt{m} d^4 \log(ed)  \right)$ are bounded from below by $1 
- 2 e^{-c_2 s^2 m \log(ed)}$. \qed}  
\end{enumerate}
\end{cor}

\subsection{The Small Ball Assumption and Bounds for $L(P)$} 
\label{sub:smallball} 
We will need the following standard lemma (see, e.g., \cite[Lemma 2.2]{RV} 
or \cite{NZ}). 
\begin{lem}
\label{smallBallTensorize} 
\noindent Let $ \xi_{1}, \ldots , \xi_{m}$ be independent random variables 
such that, for every $\eps>0$, we have 
$\mathrm{Prob} \left ( | \xi_{i} | \leq \varepsilon \right) \leq c_0 
\varepsilon$. Then there is a constant $\tilde{c}>0$ such that for every 
$\varepsilon >0$ we have $\mathrm{Prob} 
\left( \sqrt{ \xi_{1}^{2} + \cdots + \xi_{m}^{2} } \leq \varepsilon \sqrt{m} 
\right) \leq \left( \tilde{c} c_0 \varepsilon \right)^{m}$. \qed 
\end{lem}

We can then derive the following result: 
\begin{lem} \label{smallball} 
Let $P\!=\!(p_1,\ldots,p_m)$ be a random polynomial system, satisfying   
the Small Ball assumption with underlying constant $c_0$. Then there 
is a constant $\tilde{c}>0$ such that for every 
$\eps > 0$ and $x \in S^{n-1}$ we have 
$\mathrm{Prob} ( \norm{P(x)}_2 \leq \eps \sqrt{m} ) \leq (\tilde{c}c_0 
\eps)^m$. 
\end{lem}

\begin{proof}
By the Small Ball assumption on the random vectors $C_i$, and observing that 
\linebreak 
$p_i(x)=\langle C_i , \cX_i\rangle$ and  
$\norm{\cX_i}_2=1$ for all $x \in S^{n-1}$, we have 
$\mathrm{Prob}(|p_i(x)| \leq \eps )\leq c_0 \eps$. 
By Lemma \ref{smallBallTensorize} we are done.  \end{proof}

The next lemma is a variant of \cite[Claim 2.4]{Ng}. The motivation for the   
technical statement below, which introduces new parameters 
$\alpha,\beta,\gamma$, is that it is the crucial covering estimate needed to 
prove a central probability bound we'll need later: Theorem \ref{L-theorem}.  
\begin{lem}\label{Taylor}
Let $n\geq 2$, let $P:=(p_1,\ldots,p_m)$ be a system of $n$-variate homogenous 
polynomials, and assume $\norm{P}_{\infty} \leq \gamma$. Let 
$x,y \in S^{n-1}$ be mutually orthogonal vectors with $\cL(x,y) \leq \alpha$, 
and let $r\in[-1,1]$. 
Then for every $w$ with $w=x+ \beta r y + \beta^2 z$ for some $z \in B_2^{n}$ 
, we have the following inequalities:   
\begin{enumerate}
\item If $d:=\max_i d_i$ and $ 0 < \beta \leq d^{-4}$ then 
$\norm{P(w)}_2^2 \leq 8 ( \alpha^2 + (2+e^4) \beta^{4} d^4 \gamma^2)$. 
\item \scalebox{.92}[1]{If $\deg (p_i)=d$ for all $i \in [m]$, and 
$ 0 < \beta \leq d^{-2}$  then $\norm{P(w)}_2^2 \leq 8 ( \alpha^2 
+ (2+e^4) \beta^{4} d^4 \gamma^2 )$.}  
\end{enumerate}
\end{lem}

\begin{proof} 
We will prove just Assertion (1) since the proof of Assertion (2) is almost 
the same. We start with some auxiliary observations on $\norm{P}_{\infty}$: 
First note that Theorem \ref{Kellogsys} tells us that 
$\norm{P}_{\infty} \leq \gamma$ implies 
$\norm{D^{(1)}P}_{\infty} \leq d^2 \gamma$ and, similarly, 
$\norm{D^{(k)}P}_{\infty} \leq d^{2k}\gamma$ for every $k\geq 1$. Also,  
for any $w$ and $u_i \in S^{n-1}$ with $i\!\in\{1,\ldots,k\}$,  
$\norm{P}_{\infty} \leq \gamma$ and the homogeneity of the $p_i$ 
implies $\sup_{u_1,\ldots,u_k} \norm{D^{(k)} P(w) (u_1,\ldots,u_k)}_2 \leq 
\norm{w}_2^{d-k} d^{2k}\gamma$.  
These observations then yield the following inequality for 
$w=x+ \beta r y + \beta^2 z$ with $z \in B_2^{n}$, $|r|\leq 1$, 
$ \beta \leq d^{-1}$, $k=3$, and  $u_1,u_2,u_3 \in S^{n-1}$:\\  
\mbox{}\hfill 
$\norm{D^{(3)} P(w) (u_1,u_2,u_3)}_2 \leq \norm{w}_2^{d-3} d^6 \gamma \leq   
(1+2\beta)^{d-3}  d^6 \gamma $\hfill 
\mbox{}\\ 
\noindent Now, by Taylor expansion, we have the following equality:  
 
$$ p_j(w) = p_j(x)+ \langle \nabla p_j(x), \beta r y+ \beta^2 z \rangle + \frac{1}{2} (\beta r y+ \beta^2  z)^T D^{(2)}p_j(x)(\beta r y+ \beta^2 z) + \left(1+  \beta \right)^{3} \beta^3  A_j(x), $$

\noindent where $ A_j(x):= \int_0^{1} D^{(3)}p_j(x+t\norm{v}_2 v) (v,v,v)dt$  and $v=\frac{\beta r y + \beta^2 z}{\norm{\beta r y + \beta^2 z}}$.  

Breaking the second and third order terms of the expansion of $p_j(w)$ into 
pieces, we then have the following inequality: 

\[
 | p_j(w) | \leq |p_j(x)| + \beta | \langle \nabla p_j(x), y \rangle | + \beta^2 |\langle \nabla p_j(x), z \rangle| + \frac{1}{2} \beta^2 |D^{(2)}p_j(x)(y,y)| + \frac{1}{2} \beta^3 |D^{(2)}p_j(x)(y,z)|   \]
 
\[ + \frac{1}{2} \beta^3 |D^{(2)}p_j(x)(z,y)| + \frac{1}{2} \beta^4 |D^{(2)}p_j(x)(z,z)|   + (1 + \beta )^3 \beta^3 \left| A_j(x) \right|. \]

\noindent Applying the Cauchy-Schwarz Inequality to the vectors 
$(1,\beta d_j^{\frac{1}{2}}, 1, 1, 1, 1, 1, 1)$ and\linebreak  
$(\abs{p_j(x)} , d_j^{-\frac{1}{2}}|\langle \nabla p_j(x), y \rangle |, \ldots, 
(1 + \beta)^3  \beta^3 \left| A_j(x)\right|)$ then implies the following 
inequality: 
\[ p_j(w)^2 \leq (7+\beta^2 d_j)(p_j(x)^2 + d_j^{-1} 
{\langle p_j(x) , y \rangle}^2 + \beta^{4} {\langle \nabla p_j(x), z 
\rangle}^2 + \frac{1}{4} \beta^4 (D_j^{(2)}p_j(x)(y,y))^2  \]

\[ +  \frac{1}{4} \beta^6 |D^{(2)}p_j(x)(y,z)|^2 +  \frac{1}{4} \beta^6 |D^{(2)}p_j(x)(z,y)|^2 + \frac{1}{4} \beta^8 |D^{(2)}p_j(x)(z,z)|^2 + \beta^6 (1+ \beta )^6 A_j(x)^2 ) \]

\noindent We sum all these inequalities for  $j\in\{1,\ldots,m\}$.  On the 
left-hand side we  have $\norm{P(w)}_2^2$. On the right-hand side, the 
summation of the terms $p_j(x)^2 + d_j^{-1} {\langle p_j(x) , y \rangle}^2$ 
is\linebreak   
$\norm{P(x)}_2^2 + \norm{M^{-1} D^{(1)}P(x)(y)}_2^2$, and its magnitude is 
controlled by the assumption\linebreak 
$\cL(x,y) \leq \alpha$. The summations of the 
other terms are controlled by the assumption $\norm{P}_{\infty} \leq \gamma$ 
and Theorem \ref{Kellogsys}. Summing all the inequalities for  $j\in\{1,\ldots,m\}$, we have  

$$ \norm{P(w)}_2^2 \leq (7+\beta^2 d) ( \norm{P(x)}_2^2 + \norm{M^{-1} D^{(1)}P(x)(y)}_2^2 + \beta^{4} d^4 \gamma^2 +  \frac{1}{4}\beta^4 d^4 \gamma^2  $$

$$  +  \frac{1}{4} \beta^6 d^6 \gamma^2 + \frac{1}{4} \beta^6 d^6 \gamma^2+ \frac{1}{4} \beta^8 d^8 \gamma^2 +  \beta^6 (1+ \beta )^6  \sum_{j} A_j(x)^2 ) $$

\noindent The assumption  $\beta \leq d^{-4}$ implies that $\beta^8 d^8 \leq \beta^4 d^4$ and $\beta^6 d^6 \leq \beta^4 d^4$. Therefore,    

$$ \norm{P(w)}_2^2 \leq (7+\beta^2 d) ( \norm{P(x)}_2^2 + \norm{M^{-1} D^{(1)}P(x)(y)}_2^2 + \beta^{4} d^4 \gamma^2 +  \beta^4 d^4 \gamma^2 +  \beta^6 (1+ \beta )^6  \sum_{j} A_j(x)^2 ).$$  

\medskip 
\noindent Clearly  $\sum_{j\leq m} A_j(x)^2 \leq \max_{w \in V_{x,y}} \norm{D^{(3)} P(w) (u_1,u_2,u_3)}_2^2 \leq  (1+2\beta)^{2d-6} d^{12} \gamma^2$. Hence we 
have  
$\norm{P(w)}_2^2 \leq (7+\beta^2 d) ( \alpha^2 + \beta^{4} d^4 \gamma^2 + \beta^4 d^4 \gamma^2 +  (1+ 2\beta)^{2d} \beta^6 d^{12} \gamma^2 )$. 
Since $ \beta \leq d^{-4}$, we finally get  
$\norm{P(w)}_2^2 \leq (7+\beta^2 d) ( \alpha^2 +  (2+e^4) \beta^{4} 
d^4 \gamma^2 )  \leq 8 ( \alpha^{2} + (2+e^4)  \beta^{4} d^{4} \gamma^{2})$. 
\end{proof}

Lemma \ref{Taylor} controls the growth of the norm of the polynomial system 
$P=(p_1,\ldots,p_m)$ over the region 
$\{ w \in \mathbb{R}^n  : w=x+ \beta r y + \beta^2 z  ,  \abs{r} \leq 1 ,  
y \in S^{n-1}, y \perp x , z \in B_2^n  \} $.  Note in particular that 
we are using {\em cylindrical neighborhoods} instead of ball neighborhoods. 
This is because we have found that (a) our approach truly requires us to go to 
order $3$ in the underlying Taylor expansion and (b) cylindrical neighborhoods 
allow us to properly take contributions from tangential directions, and thus 
higher derivatives, into account. 

We already had a probabilistic estimate in Lemma \ref{smallball} that said  
that for any $w$ with $\norm{w}_2 \geq 1$, the probability of $\norm{P(w)}_2$ 
being smaller than $\varepsilon \sqrt{m}$ is less than $\varepsilon^{m}$ up to 
some universal constants. The controlled growth provided by Lemma \ref{Taylor} 
holds for a region with a certain volume, which will ultimately contradict the 
probabilistic estimates provided by Lemma \ref{smallball}. This will be 
the main trick behind the proof of the following theorem.  

\begin{thm}\label{L-theorem}
Let $n\geq 2$ and let $P:=(p_1,\ldots, p_m)$ be a system of random homogenous 
$n$-variate polynomials such that 
$p_j(x)=\sum_{\abs{a}=d_j} c_{j,a} \sqrt{\binom{d_i}{a}} 
x^{a}$ where $C_j=(c_{j,a})_{\abs{a}=d_j}$ are random 
vectors satisfying the Small Ball assumption with underlying constant 
$c_0$. Let $\alpha, \gamma>0$, $d := \max_{i} d_i$, and assume 
$\alpha  \leq \gamma \min\left\{ d^{-6}, d^2/n \right\}$. Then
$$ \mathrm{Prob} ( L(P)\leq \alpha ) \leq  \mathrm{Prob} \left( \|P\|_{\infty} 
\geq \gamma\right)  +   \alpha^{\frac{3}{2}+m-n}  \sqrt{n} ( \gamma d^2)^{n-\frac{3}{2}} \left(  \frac{C c_0 }{ \sqrt{m}} \right)^{ m } $$ 

\noindent where $C$ is a universal constant. 
\end{thm}

\begin{proof}
\noindent We assume the hypotheses of Assertion (1): 
Let $ \alpha,\gamma >0$ and 
$\beta \leq d^{-4}$. Let $ {\bf B:}= \{P\; | \;  \|P\|_{\infty} \leq 
\gamma\}$ and let \\ 
\mbox{}\hfill $ {\bf L} := \{P\; | \; L(P) \leq \alpha\} = 
\{ P\; | \; \text{There exist } x,y\!\in\!S^{n-1} \text{ with }  
x\perp y \text{ and }  \cL(x,y) \leq  \alpha \} $.\hfill\mbox{}\\ 
Let $\Gamma:=8(\alpha^{2} + (5+e^4) \beta^{4} d^{4} \gamma^{2})$ and 
let $B^n_2$ denote the unit $\ell_2$-ball in $\R^n$. 
Lemma \ref{Taylor} 
implies that if the event ${\bf B}\cap {\bf L}$ occurs then there exists a 
set $$V_{x,y}:= \{ w\in \mathbb R^{n} :w= x+ \beta r y+ \beta^2 z , \abs{r} \leq 1, z \perp y , z \in B_2^{n} \} \setminus B_{2}^{n} $$
such that $ \| P(w) \|_{2}^{2} \leq \Gamma$ for every $w$ in this set. 
Let $V:= \mathrm{Vol}\! \left( V_{x,y} \right)$. Note that for $ w\in V_{x,y}$ 
we have $\norm{w}_2^2 = \norm{x+\beta^2z}_2^2  + \norm{\beta y}_2^2 \leq 1 + 
4 \beta^2$. Hence we have $\norm{w}_2 \leq 1 + 2 \beta^2$. Since\linebreak 
$V_{x,y} \subseteq (1+ 2\beta^2 ) B_{2}^{n} 
\setminus B_{2}^{n} $, we have showed that\\ 
\mbox{}\hfill ${\bf B}\cap {\bf L}\subseteq 
\left\{P\; | \; \mathrm{Vol}\!\left(\{x\in (1+ 2 \beta^2 ) B_{2}^{n}\setminus 
B_{2}^{n} \; | \; \|P(x) \|_{2} \leq \Gamma \}\right)  \geq V \right\}$.\hfill
\mbox{} \\  
Using Markov's Inequality, Fubini's Theorem, and Lemma \ref{smallball}, we 
can estimate the probability of this event. Indeed, \\ 
$\mathrm{Prob}\left(  \mathrm{Vol}\! \left( \{ x\in (1+ 2 \beta^2) B_{2}^{n} \setminus 
B_{2}^{n}  : \|P(x) \|_{2} \leq \Gamma \}  \right)  \geq V  \right)$\hfill\mbox{} 
\begin{eqnarray*} 
& \leq & \frac{1}{ V}  \mathbb{E} \mathrm{Vol}\!\left( \{ x\in (1+ 2 \beta^2) B_{2}^{n} \setminus B_{2}^{n}  : \|P(x) \|_{2}^{2} \leq \Gamma \}  \right) \\ 
& \leq & \frac{1}{V} \int_{ (1+ 2\beta^2) B_{2}^{n} \setminus B_{2}^{n} } 
\mathrm{Prob} \left(\|P(x)\|_{2}^{2} \leq \Gamma \right) dx \\ 
& \leq & \frac{ \mathrm{Vol}\!\left( (1+ 2\beta^2) B_{2}^{n} \setminus B_{2}^{n} \right) }{V} 
\max_{x\in (1+ 2\beta^2) B_{2}^{n} \setminus B_{2}^{n} }  \mathrm{Prob} 
\left( \| P(x)  \|_{2}^{2} \leq \Gamma \right).   
\end{eqnarray*} 
 
Now recall that $\vol(B_{2}^{n})= \frac{\pi^{n/2}}{ \Gamma\left(\frac{n}{2}+1
\right)}$. Then $\frac{\vol(B_{2}^{n})}{\vol(B_{2}^{n-1})} \leq \frac{c'}
{\sqrt{n}}$ for some constant $c'>0$.  Having assumed  
that $\beta^2 \leq \frac{1}{n}$ we obtain 
$(1+2\beta^2)^n \leq 1 + 2n \beta^2 $, and 
we see that 
$$  \frac{\vol((1+ 2\beta^2) B_{2}^{n} \setminus B_{2}^{n}) }{V} 
\leq \frac{\vol(B_{2}^{n})\left( (1+ 2\beta^2)^{n} -1\right) }{ \beta 
(\beta^2)^{n-1} \vol(B_{2}^{n-1})} \leq c \sqrt{n} \beta \beta^{2-2n} ,$$
for some absolute constant $c>0$. Note that here, for a lower bound on $V$, we 
used the fact that $V_{x,y}$ contains more than half of a cylinder with 
base having radius $\beta^2$ and height $2\beta$. 

Writing $ \tilde{x} := \frac{ x}{ \|x\|_{2}}$
for any $ x\neq 0$ we then obtain, for $z\notin B_{2}^{n}$, that 
$$ \| P (z) \|_{2}^{2} = \sum_{j=1}^{m} | p_{j} (z) |^{2} = \sum_{j=1}^{m} 
|p_j (\tilde{z}) |^{2} \| z\|_{2}^{2 d_j} \geq  \sum_{j=1}^{m} | p_{j} 
(\tilde{z}) |^{2} = \| P ( \tilde{z}) \|_{2}^{2}.$$ 
\noindent This implies, via Lemma \ref{smallball}, that for every 
$w\!\in\!(1+ 2\beta^2) B_{2}^{n} \setminus 
B_{2}^{n}$ we have  
$$ \mathrm{Prob} \left( \|P(w) \|_{2}^{2} \leq \Gamma\right) \leq  \mathrm{Prob} \left( \|P(\tilde{w}) \|_{2}^{2} \leq \Gamma\right) \leq 
\left( c c_0 \sqrt{\frac{\Gamma}{m}}\right)^{ m }. $$

So we conclude that\\  
\scalebox{.92}[1]{$\mathrm{Prob} ( L(P)\leq \alpha ) \leq \mathrm{Prob} 
\left(\|P\|_{\infty} 
\geq \gamma\right) + \mathrm{Prob} \left({\bf  B} \cap {\bf L}\right) \leq  
\mathrm{Prob} \left( \|P \|_{\infty} \geq \gamma\right)  + c \sqrt{n} 
\beta \beta^{2-2n} \left( c c_0 \sqrt{\frac{\Gamma}{m}} \right)^m$.}\linebreak  
Recall that $ \Gamma= 8(\alpha^{2} + (5+e^4)\beta^{4} d^{4} \gamma^{2})$. 
Setting $ \beta^2 := \frac{ \alpha }{ \gamma  d^{2}}$,  
our assumption  $\alpha  \leq \gamma \min\left\{ d^{-6}, d^2/n \right\}$ and 
our choice of $\beta$ then imply that 
$\Gamma= C \alpha^2$ for some constant $C$. So we obtain  
$$ \mathrm{Prob} ( L(P)\leq \alpha ) \leq  \mathrm{Prob} \left( \|P\|_{\infty} 
\geq \gamma\right)  + c \sqrt{n} \left( \frac{ \alpha }{ \gamma d^{2}}
\right)^{\frac{3}{2}-n} \left( \frac{C c_0  \alpha }{ \sqrt{m}}\right)^{ m } $$ 

\noindent 
and our proof is complete. \end{proof}

\medskip

\subsection{The Condition Number Theorem and its Consequences} 
\label{sec:main1}
We will now need bounds for the Weyl-Bombieri norms of polynomial systems. 
Note that, with\\ 
\mbox{}\hfill $p_j(x)\!=\!\sum\limits_{\alpha_1+\cdots+\alpha_n=d_j} 
\sqrt{\binom{d_j}{\alpha}} c_{j,\alpha} x^{\alpha}$,\hfill\mbox{}\\ 
we have $\|p_j\|_{W} := \| (c_{j,\alpha})_\alpha \|_{2}$ 
for $j\in\{1,\ldots,m\}$. The following lemma, providing large deviation 
estimates for the Euclidean norm, is standard and follows, for instance, from 
Theorem \ref{Bernstein}. 
\begin{lem} \label{AG}
\noindent 
There is a universal constant $c'\!>\!0$ such that for any random 
$n$-variate polynomial system $P\!=\!(p_1,\ldots,p_m)$ satisfying the 
Centering and Sub-Gaussian assumptions, with underlying constant $K$, 
$j\in\{1,\ldots,m\}$, $N_{j}:=\binom{n+d_j-1}{d_j}$, 
$N:=\sum_{j=1}^{m} N_{j}$, and $t\geq 1$, we have 
\begin{enumerate} 
\item $ \mathrm{Prob}\left(  \|p_{j}\|_W \geq c't K \sqrt{N_{j}}  \right) \leq 
e^{ -t^{2} N_j} $

\item $ \mathrm{Prob}\left(  \|P\|_W \geq c't K \sqrt{N}  \right) \leq 
e^{ -t^{2} N} .$  

\end{enumerate}
\end{lem}

We are now ready to prove our main theorem on the condition number of 
random polynomial systems. 

\begin{thm}\label{gcondition} There are universal constants $A,c\!>\!0$ 
such that the following hold: 
Let $P=(p_1,\ldots,p_m)$ be a system of homogenous random polynomials with 
$p_j(x)=\sum_{|\alpha|=d_j} c_{j,\alpha}\sqrt{\binom{d_j}{\alpha}} 
x^{\alpha}$\\ 
\linebreak 
\scalebox{.92}[1]{and let
$C_j=\left(c_{j,\alpha}\right)_{|\alpha|=d_j}$ be independent random 
vectors satisfying the Sub-Gaussian and Small Ball}\linebreak 
\scalebox{.91}[1]{assumptions, with respective 
underlying constants $K$ and $c_0$. Assume $n \geq 2$ and let 
$d:=\max_j \deg p_j$.}\linebreak 
Then, setting 
$M\!:=\!\sqrt{\frac{N}{m}} (K c_0 C)^{\frac{m}{m-n+\frac{3}{2}}} 
(3d^2\log(ed))^{\frac{n-\frac{3}{2}}{m-n+\frac{3}{2}}}  n^{\frac{1}{2m-2n+3}}  
\max \left\{ d^{6}, \frac{n}{d^2} \right\}$,  
we have two cases: 
\begin{enumerate}
\item If $ N \geq  m \log(ed)$ then 
$ \mathrm{Prob}(\tilde{\kappa}(P) \geq t M)$ is bounded from above by   
$$\begin{cases}
  \frac{3}{ t^{m-n+\frac{3}{2}}} &\mbox{if } 1\leq t\leq e^{\frac{m\log(ed)}{ m-n+\frac{3}{2}}}\\
 \frac{3}{t^{m-n+\frac{3}{2}}} \left( \frac{(m-n+\frac{3}{2})\log t}{m\log(ed)}\right)^{\frac{n-\frac{3}{2}}{2}}  &\mbox{if } e^{ \frac{ m \log(ed)}{m-n+\frac{3}{2}}} \leq t \leq e^{\frac{N}{m-n+\frac{3}{2}}} \\
\frac{3}{t^{m-n+\frac{3}{2}}} \left( \frac{(m-n+\frac{3}{2})\log t}{N}\right)^{\frac{m}{2}}  \left( \frac{N}{m\log(ed)}\right)^{\frac{n-\frac{3}{2}}{2}}  &\mbox{if }  e^{\frac{N}{m-n+\frac{3}{2}}} \leq t  \end{cases}
 $$
\item  If $ N \leq  m \log(ed)$ then 
$\mathrm{Prob}(\tilde{\kappa}(P) \geq t M )$ is bounded from above by 
$$ \begin{cases}
  \frac{3}{ t^{m-n+\frac{3}{2}}} &\mbox{if } 1\leq t\leq e^{\frac{N}{ m-n+\frac{3}{2}}} \\
\frac{3}{t^{m-n+\frac{3}{2}}} \left( \frac{(m-n+\frac{3}{2})\log t}{N}\right)^{\frac{m}{2}}  
&\mbox{if }  e^{\frac{N}{m-n+\frac{3}{2}}} \leq t  \end{cases}
 $$
\end{enumerate}
\end{thm}

\begin{proof}
Recall that 
$\tilde{\kappa}(P)\!=\!\frac{\|P\|_W}
{L(P)}$. Note that if $u\!>\!0$, and the inequalities 
$\|P\|_W\!\leq\!ucK\sqrt{N}$ and $L(P)\!\geq\!\frac{ucK\sqrt{N}}{tM}$ 
hold, then we clearly have $\tilde{\kappa}(P)\!\leq\!tM$. 
In particular, $u\!>\!0$ implies that 
$$ \mathrm{Prob} \left( \tilde{\kappa}(P) \geq t M  \right) \leq \mathrm{Prob} 
\left( \norm{P}_W \geq u c K \sqrt{N} \right) + \mathrm{Prob} 
\left(L(P) \leq \frac{ u c K \sqrt{N} }{t M}
\right).$$
\noindent 
Our proof will then reduce to optimizing $u$ over the 
various domains of $t$. 

Toward this end, note that Lemma \ref{AG} provides a large deviation estimate 
for the Weyl norm of our polynomial system. So, to bound $\mathrm{Prob} 
\left( \norm{P}_W \geq u c K \sqrt{N} \right)$ from above, we need to use 
Lemma \ref{AG} with the parameter $u$. As for the other summand in the 
upper bound for 
$\mathrm{Prob} \left( \tilde{\kappa}(P) \geq t M  \right)$,   
Theorem \ref{L-theorem} provides an upper bound for 
$ \mathrm{Prob}  \left(L(P) \leq \frac{ u c K \sqrt{N} }{t M} \right)$. 

However, the upper bound provided by Theorem \ref{L-theorem} involves the 
quantity $\mathrm{Prob} \left( \norm{P}_{\infty}   \geq \gamma \right)$. 
Therefore, in order to bound $ \mathrm{Prob}\!\left(L(P) \leq 
\frac{ u c K \sqrt{N} }{t M} \right)$, we will need to use Theorem 
\ref{L-theorem} 
together with Lemma \ref{operatornorm}. In particular, we will set 
$\alpha:=\frac{ ucK\sqrt{N}}{ t M}  $ and $\gamma:=3 s K\sqrt{m} \log (ed)$  
in Theorem \ref{L-theorem} and Lemma \ref{operatornorm},  
and then optimize the parameters $u$, $s$, and $t$ at the final step of the 
proof. 

Now let us check if the assumptions of Theorem \ref{L-theorem} are 
satisfied: We have that $s\geq 1$, $u\geq 1$, and (since $\alpha \leq 
\min\left\{ d^{-6}, \frac{d^{2}}{n}\right\}  \gamma$) we have 

$$  \frac{ ucK\sqrt{N}}{ t M} \leq   3 s K\sqrt{m} \log (ed)  \min \left\{ d^{-6}, 
\frac{d^{2}}{n}\right\}.  $$

\noindent So $\frac{ uc \sqrt{N} } { \sqrt{m} \log(ed) } \leq 3 s t M  
\min \left\{ d^{-6}, \frac{d^{2}}{n}\right\}$ 
and we thus obtain 

$$  (\ast) \ \ \ \ \ \ \ \ \ \ \ \ \ 
\frac{uc}{\log(ed)} \sqrt{\frac{N}{m}}  \leq  3st  \sqrt{\frac{N}{m}} (K c_0 C)^{\frac{m}{m-n+\frac{3}{2}}} (3d^2\log(ed))^{\frac{n-\frac{3}{2}}{m-n+\frac{3}{2}}}   n^{\frac{1}{2m-2n+3}}.   $$

\noindent Since $Kc_o \geq \frac{1}{4}$,  
the inequality ({\bf $\ast$}) holds if $u\!\leq\!s$, $t\!\geq\!1$, 
and we take the constant $C$ from Theorem \ref{L-theorem} to be at 
least $4$. Under the preceding restrictions we then have that 
$Q:= \mathrm{Prob} ( \tilde{\kappa} (p) \geq t M)$ implies 
$$ Q \leq   \left( \frac{ ucK\sqrt{N}}{ t M}  \right)^{\frac{3}{2}+m-n}  
\sqrt{n} (3 s K\sqrt{m} \log (ed) d^2)^{n-\frac{3}{2}}   \left(  \frac{C c_0 }
{ \sqrt{m}} \right)^{ m }  + e^{ - c_{2} s^{2} m \log(ed)} + e^{ - u^{2} N}$$ 
or, $ Q \leq  \frac{ u^{m-n+\frac{3}{2}} s^{n-\frac{3}{2}}}{t^{m-n
+\frac{3}{2}}} + e^{-c_{2} s^{2} m\log(ed)} 
+ e^{-u^{2} N}$ for some suitable $c_2\!>\!0$. 

We now consider the case where $ N \geq  m \log(ed)$. 
If $1\leq t \leq e^{ \frac{c_{1} m \log(ed)}{ m-n+\frac{3}{2}}}$ then we set  
$u\!=\!s\!=\!1$,\linebreak 
noting that ({\bf $\ast$}) is satisfied.  We then obtain\\ 
\mbox{}\hfill $Q \leq  \frac{1}{ t^{m-n+\frac{3}{2}} } 
+ e^{- c_{2} m\log(ed)} +e^{-N} \leq  \frac{3}{t^{m-n+\frac{3}{2}}}$, 
\hfill\mbox{}   

provided $c_2\!\geq\!1$. 
In the case where $e^{ \frac{ m \log(ed)}{m-n+\frac{3}{2}}} \leq t \leq 
e^{\frac{N}{m-n+\frac{3}{2}}} $ we choose $ u=1$ and $ s:= \sqrt{\frac{(m-n+\frac{3}{2})\log t}
{m\log(ed)}} \geq 1$. (Note that $u\leq s$). These choices then yield 
$$ Q\leq \frac{ 1 }{t^{m-n+\frac{3}{2}}} \left( \frac{(m-n+\frac{3}{2})\log t}{m\log(ed)}
\right)^{n-\frac{3}{2}} + \frac{1}{ t^{c_{2}(m-n+\frac{3}{2})}} + e^{ -N} \leq \frac{ 3}
{t^{m-n+\frac{3}{2}}} \left( \frac{(m-n+\frac{3}{2})\log t}{m\log(ed)}\right)^{\frac{n-\frac{3}{2}}
{2}}.$$

In the case where $e^{\frac{N}{m-n+\frac{3}{2}}} \leq t$, we choose $ s:=\sqrt{ \frac{ ({\log t }) (m-n+\frac{3}{2})}{m\log(ed)}} $ and $ u:= \sqrt{ \frac{(m-n+\frac{3}{2})\log t}{N}}$. 
(Note that $u\leq s$ also in this case). So we get 
$$ Q\leq \frac{ 1}{t^{m-n+\frac{3}{2}}} \left( \frac{(m-n+\frac{3}{2})\log t}{N}
\right)^{\frac{m}{2}}  \left( \frac{N}{m\log(ed)}\right)^{\frac{n-\frac{3}{2}}{2}} + 
\frac{1}{ t^{c_2(m-n+\frac{3}{2})}} + \frac{1}{t^{m-n+\frac{3}{2}}} $$
$$ \leq \frac{3  }{t^{m-n+\frac{3}{2}}} \left( \frac{(m-n+\frac{3}{2})\log t}{N}\right)^{\frac{m}{2}}  \left( \frac{N}{m\log(ed)}\right)^{\frac{n-\frac{3}{2}}{2}}.$$ 

We consider now the case where $ N \leq m \log(ed)$. When  
$1\leq t \leq e^{\frac{N}{ m-n+\frac{3}{2}}}$ we choose $s=1$ and $u=1$ to 
obtain $Q \leq   \frac{1}{ t^{m-n+\frac{3}{2}}} + e^{- c_{2} m\log(ed)}+e^{-N} 
\leq \frac{ 3 }{ t^{m-n+\frac{3}{2}}}$ as before. 
In the case $t\geq e^{\frac{N}{m-n+\frac{3}{2}}}$, we choose $ s=u:= 
\sqrt{\frac{(m-n+\frac{3}{2})\log t}{N}} $. Note that again ({\bf $\ast$}) is 
satisfied and, with these choices, we get 
$$  Q\leq \frac{ 1 }{t^{m-n+\frac{3}{2}}} \left( \frac{(m-n+\frac{3}{2})\log t}{N}
\right)^{\frac{m}{2}}  + \frac{1}{ t^{c_{2}(m-n+\frac{3}{2})m\log(ed)/N}}+ 
\frac{1}{t^{m-n+\frac{3}{2}}}$$  
$$\leq \frac{3 }{t^{m-n+\frac{3}{2}}} \left( \frac{  ({\log t })(m-n+\frac{3}{2})}{N}
\right)^{\frac{m}{2}}.$$  \end{proof} 

\begin{thm} \label{expectation}
Let $P$ be a random polynomial system as in Theorem \ref{gcondition}, let 
$d\!:=\!\max_j \deg p_j$, and let 
$M$ be as defined in Theorem \ref{gcondition}. Set 
$$ \delta_1 :=   
\frac{q\sqrt{\pi n}}{ m-n+\frac{3}{2}}  \left( \frac{n-\frac{3}{2}}{ 2em\log(ed)}\right)^{\frac{n-\frac{3}{2}}{2}} \frac{1}{ \left( 1- \frac{q}{m-n+\frac{3}{2}}\right)^{\frac{n}{2}}} \ \ \ \ 
\text{ and } $$
$$ \delta_2 := \left( \frac{m}{N}\right)^{ \frac{ m-n+\frac{3}{2}}{ 2}} 
\frac{q\sqrt{\pi m}e^{-\frac{m}{2}}}{
\left(m-n+\frac{3}{2}-q\right)
\left( 1- \frac{q}{ m-n+\frac{3}{2}}\right)^{\frac{m}{2}} 
(\log(ed))^{\frac{n}{2}-1}}. $$ 

\noindent We then have the following estimates:  
 \begin{enumerate}
\item If $ N \geq m \log(ed)$ and $q\!\in\!(0,m-n+\frac{3}{2})$  then 
$$ \left( \mathbb{E}({\tilde{\kappa}(P)}^{q})  \right)^{\frac{1}{ q}}  \leq {M} 
\left( 1+ \frac{q}{m-n-q+2}  + \delta_{1} + \delta_{2} \right)^{\frac{1}{q}}.$$ 
\noindent In particular, 
$q\!\in\!\left(0,(m-n+\frac{3}{2})\left(1-\frac{ 1}{ 2\log(ed)} \right)\right] 
\Longrightarrow \left( \mathbb{E}({\tilde{\kappa}(P)}^{q})
\right)^{\frac{1}{ q}}  \leq {M} 
\left( \frac{3m\log(ed)}{ n}\right) ^{\frac{1}{q}}$, and 
$q\!\in\!\left(0,\frac{m-n+\frac{3}{2}}{2}\right] \Longrightarrow 
\left( \mathbb{E}({\tilde{\kappa}(P)}^{q})\right)^{\frac{1}{ q}} 
\leq 4^{1/q}M$.\\  
Furthermore, $\mathbb{E}(\log{ \tilde{\kappa}(P)}) \leq 1+ \log M$. 
\item If $ N \leq m \log(ed)$, then 
$\left(\mathbb{E}({\tilde{\kappa}(P)}^{q})  
\right)^{\frac{1}{ q}}\leq M \left(1+ \frac{q}{ m-n-q+\frac{3}{2}} + \delta_{2}
\right)^{\frac{1}{q}}$.\\  
In particular, $q\!\in\!\left(0, (m-n+\frac{3}{2}) \left( 1- 
\frac{m}{eN}\right)\right]\Longrightarrow 
\left( \mathbb{E}({\tilde{\kappa}(P)}^{q})  
\right)^{\frac{1}{ q}} \leq M \left( \frac{3m\log(ed)}{ n}\right)^{\frac{1}
{ q}}$ and $q\in\left(0,\frac{m-n+\frac{3}{2}}{2}\right]\Longrightarrow 
\left( \mathbb{E}({\tilde{\kappa}(P)}^{q})  \right)^{\frac{1}{ q}}  \leq 
4^{1/q}M$.\\  
Furthermore, $\mathbb{E}(\log{ \tilde{\kappa}(P)}) \leq 1+ \log M$. 
\end{enumerate}
\end{thm}

\begin{proof}
\noindent Set \ \ $\Lambda_{1}:= \left( \frac{ m-n+\frac{3}{2}}{ m\log{ed}}
\right)^{\frac{n-\frac{3}{2}}{2}}$, \ \ $\Lambda_{2} :=  
\left( \frac{ m-n+\frac{3}{2}}{ N}\right)^{\frac{m}{2}} 
\left( \frac{N}{ m\log{ed}}\right)^{\frac{n-\frac{3}{2}}{2}}$,\\ 
$$ r:= m-n-q+\frac{5}{2}, \ a_{1} := \frac{m\log{ed}}{ m-n+\frac{3}{2}} , \text{ and } 
a_{2} := \frac{N}{ m-n+\frac{3}{2}} . $$ 
\noindent Note that we have $r\geq 1$ by construction. Using Theorem 
\ref{gcondition} and the formula  
$$\mathbb{E}((\tilde{\kappa}(P))^q) = q \int_0^{\infty} t^{q-1}\mathrm{Prob} 
\left(\tilde{\kappa}(p) \geq t \right) d t $$ 
(which follows from the definition of expectation), we have that \\ 
\mbox{}\hfill 
$\displaystyle{\mathbb{E}((\tilde{\kappa}(P))^q)  \leq M^{q} \left( 1 + 
q\int_{1}^{\infty} t^{q-1} \mathrm{Prob} \left(\tilde{\kappa}(p) \geq t  
M\right) d t \right)}$,\hfill\mbox{}\\  
or \ $\displaystyle{\frac{  \mathbb{E}((\tilde{\kappa}(P))^q)}{ M^{q}} \leq  1+ 
q\int_{1}^{ e^{ a_{1}} } \frac{1}{ t^{r} }  dt + q\Lambda_{1}
\int_{e^{a_{1}}}^{ e^{ a_{2}}} \frac{ (\log t )^{ \frac{n-\frac{3}{2}}{2}}}
{t^{r} } dt + q\Lambda_{2}\int_{e^{a_{2}}}^{\infty} 
\frac{(\log t)^{\frac{m}{2}}}
{ t^{r} } d t}$. We will give upper bounds for the last three integrals. First 
note that
$$ q\int_{1}^{ e^{ a_{1}} } \frac{1}{ t^{r} }  dt = \frac{q}{ r-1} 
\left(1-e^{(r-1) a_{1}}\right) \leq  \frac{q}{ r-1} . $$
Also, we have that
$$ q\Lambda_{1}\int_{e^{a_{1}}}^{ e^{ a_{2}}} 
\frac{(\log t)^{\frac{n-\frac{3}{2}}{2}}}{ t^{r} } d t  
 = q\Lambda_{1}\int_{a_{1}}^{a_{2}} t^{ \frac{n-\frac{3}{2}}{2}} e^{ (r-1) t} 
  d t = \frac{q \Lambda_{1}} { (r-1)^{\frac{n}{2}}} 
   \int_{a_{1}(r-1)}^{ a_{2}(r-1)} t^{ \frac{n-\frac{3}{2}}{2}} e^{-t} d t  $$
$$ \leq \frac{q \Lambda_{1}}{(r-1)^{\frac{n}{2}-\frac{1}{4}}} 
    \Gamma\left(\frac{n}{2}-\frac{1}{4}\right) 
  \leq \frac{q\sqrt{\pi n}}{ m-n+\frac{3}{2}}  
   \left( \frac{n-\frac{3}{2}}{ 2em\log(ed)}
\right)^{\frac{n}{2}-\frac{3}{4}} \frac{1}{ \left( 1- \frac{q}{m-n+\frac{3}{2}}
\right)^{\frac{n}{2}-\frac{1}{4}}}.$$ 
Finally, we check that 
$$ q\Lambda_{2}\int_{e^{a_{2}}}^{\infty} \frac{ (\log t )^{ \frac{m}{2}}}{t^r} 
d t  = q\Lambda_{2}\int_{a_{2}}^{\infty} t^{ \frac{m}{2}} e^{ (r-1) t} d t = 
\frac{q \Lambda_{2}} { (r-1)^{\frac{m}{2}+1}} \int_{a_{2}(r-1)}^{\infty} 
t^{\frac{m}{2}} e^{-t} d t  $$ 
$$ \leq \frac{q \Lambda_{2}} { (r-1)^{\frac{m}{2}+1}} 
 \Gamma\left(\frac{m}{2}+1\right) \leq \frac{q\sqrt{\pi m}}
 {(m-n-q+\frac{3}{2})^{\frac{m}{2}+1}} 
 \left( \frac{ m(m-n+\frac{3}{2})}{e N}\right)^{\frac{m}{2}} 
 \left( \frac{N}{ m\log{ed}}\right)^{\frac{n}{2}-\frac{3}{4}} $$
$$ =\left( \frac{m}{N}\right)^{\frac{m}{2}-\frac{n}{2}+\frac{3}{4}} 
 \frac{1}{\left(1- \frac{q}{ m-n+\frac{3}{2}}\right)^{\frac{m}{2}}} 
 \cdot \frac{q e^{-m/2} \sqrt{\pi m}}{ m-n+\frac{3}{2} -q} \cdot 
 \frac{1}{(\log(ed))^{\frac{n}{2}-\frac{3}{4}}}  . $$ 
Note that if $ q\leq  (m-n+\frac{3}{2}) \left( 1- \frac{ 1}{ 2\log(ed)}\right)$ 
then $\delta_{1} , \delta_{2} \leq 1$. 

For the case  $N \leq m \log(ed)$, working as before, we get that 
$$ \frac{  \mathbb{E}((\tilde{\kappa}(P))^q)}{ M^{q}} \leq  1+ q\int_{1}^{ e^{ a_{2}} } \frac{1}{ t^{r} }  dt +  q \Lambda_2 \int_{e^{a_{2}}}^{\infty} \frac{ (\log t )^{ \frac{m}{2}}}{ t^{r} } d t  \leq 1+ \frac{q}{ r-1} + \delta_{2}. $$ 
In the case $ N\leq m\log(ed)$ we have $\delta_{2} \leq \frac{ \sqrt{\pi m}q}
{ m-n+\frac{3}{2}} \left( \frac{m}{e N}\right)^{\frac{m}{2}} 
\frac{1}{\left(1-\frac{q}{ m-n+\frac{3}{2}}\right)^{\frac{m}{2}+1}}$. In 
particular, for this case, it easily follows that $q\leq (m-n+\frac{3}{2}) 
\left(1-\frac{m}{N}\right)$ implies $\delta_{2} \leq 1$.  
\end{proof}

Note that if $m\!=\!n-1$, $n\!\geq\!3$, and $d\!\geq\!2$, then 
$N\!\geq\!m\log(ed)$ and, in this case, 
it is easy to check that $(\ast)$ still holds even if we reduce $M$ by 
deleting its factor of $\max\!\left\{d^6,\frac{n}{d^2}\right\}$. So then, for 
the important case $m=n-1$, our main theorems immediately 
admit the following refined form:  
\begin{cor} \label{square} There are universal constants $A,c\!>\!0$ 
such that if $P$ is any random polynomial system as in Theorem 
\ref{gcondition}, but with $m=n-1$, $n \geq 3$, $d:=\max_j \deg p_j$, 
$d \geq 2$, and $M:=\sqrt{N} (K c_0 C)^{2(n-1)} 
(3d^2\log(ed))^{2n-3} \sqrt{n}$ instead, then we have:\\  
\mbox{}\hfill $\mathrm{Prob}(\tilde{\kappa}(P) \geq t M) \leq  \begin{cases}
  3 t^{-\frac{1}{2}} &\mbox{if } 1\leq t\leq e^{ 2(n-1)\log{(ed)}} \\
  3 t^{-\frac{1}{2}} \left( \frac{{\log{t}}}{2 (n-1)\log{(ed)}}\right)^{\frac{n-\frac{3}{2}}{2}}  &\mbox{if } e^{ 2 (n-1) \log{(ed)} } \leq t \leq e^{2N} \\
  3 t^{-\frac{1}{2}} \left( \frac{{\log{t}}}{2N} \right)^{\frac{1}{4}}  
\left( \frac{{\log{t}}}{2(n-1)\log{(ed)}} \right)^{\frac{n-\frac{3}{2}}{2}}  &\mbox{if }  
e^{2N} \leq t  \end{cases}$,\hfill\mbox{}\\ 
and, for all $q\!\in\!\left(0, \frac{1}{2}- \frac{1}{ 4\log{(ed)}}\right]$, we have 
$\left( \mathbb{E}({\tilde{\kappa}(P)}^{q})  \right)^{\frac{1}{ q}}\leq M 
e^{\frac{1}{ q}}$. \\ Furthermore, $\mathbb{E}(\log \tilde{\kappa}(P)) \leq 
1 + \log M$. \qed 
\end{cor}

We are now ready to prove Corollary \ref{cor:main1} from the introduction. 

\medskip 
\noindent 
\scalebox{.95}[1]{{\bf Proof of Corollary \ref{cor:main1}:} From Corollary 
\ref{square}, Bound (2) follows immediately, and Bound (1)}\linebreak 
is clearly true for the smaller domain of $t$. So let us now 
consider $t=xe^{2(n-1)\log(ed)}$ with \mbox{$x \geq 1$}.\linebreak 
\scalebox{.91}[1]{Clearly, 
\mbox{$\left( \frac{\log t }{2(n-1)\log(ed)} \right)^{\frac{n}{2}-\frac{3}{4}} 
= \left(1 + \frac{\log x}{2(n-1)\log(ed)} \right)^{\frac{n}{2}-\frac{3}{4}}$}, 
and thus $\left( \frac{\log t}{2(n-1)\log(ed)} 
\right)^{\frac{n}{2}-\frac{3}{4}}  <  
e^{\frac{\log x}{4\log(ed)}}=x^{\frac{1}{4\log(ed)}}$.} 
Since $x= \frac{t}{e^{2(n-1)\log(ed)}}$ we thus obtain  
$3t^{-\frac{1}{2}} \left( \frac{\log t}{2(n-1)\log(ed)} 
\right)^{\frac{n-\frac{3}{2}}{2}} \leq 3t^{-\frac{1}{2}} 
\left( \frac{t}{e^{2(n-1)\log(ed)}} \right)^{\frac{1}{4\log(ed)}}$. 
Renormalizing the pair $(M,t)$ (since the $M$ from Corollary \ref{square} 
is larger than the $M$ from Corollary \ref{cor:main1} by a factor of 
$A$), we are done. \qed

\subsection{On the Optimality of  Condition Number Estimates} 
\label{sub:lower} 
As mentioned in the introduction, to establish a lower bound we need one more 
assumption on the randomness. For the convenience of the reader, we 
recall our earlier Euclidean Small Ball assumption. 

\medskip
\noindent 
(\textbf{Euclidean Small Ball}) {\em There is a constant $\tilde{c}_0>0$ such that for 
each $j\!\in\!\{1,\ldots,m\}$ and $\eps\!>\!0$ we have  
$\mathrm{Prob}\left( \| C_j\|_{2} \leq \varepsilon \sqrt{N_{j}} \right)  
\leq ( \tilde{c_0} \varepsilon)^{N_{j}}$.}  

\medskip
\noindent We will need an extension of Lemma \ref{smallBallTensorize}: Lemma 
\ref{Aniso-p} below (see also \cite[Thm.\ 1.5 \& Cor.\ 8.6]{RV-1}). 
Toward this end, for any matrix $T:= (t_{i,j})_{1\leq i, j \leq m}$, write 
$\|T\|_{HS}$ for the {\em Hilbert-Schmidt norm} of $T$ and $ \|T\|_{op}$ for 
the {\em operator norm} of $T$, i.e., 
$$ \|T \|_{HS} := \left( \sum_{i, j=1}^{m} t_{i,j}^{2} \right)^{\frac{1}{ 2}}  
\ \text{  and  } \ \|T\|_{op}:= \max_{\theta \in S^{n-1}} \| T \theta \|_{2}.$$

\begin{lem}\label{Aniso-p}
\noindent Let $\xi_1,\ldots,\xi_m$ be independent random variables 
satisfying 
$\mathrm{Prob} \left ( \xi_{i} \leq \varepsilon \right) \leq c_0 \varepsilon$ 
for all $i\!\in\!\{1,\ldots,m\}$ and $\eps\!>\!0$. 
Let $ \xi := (\xi_{1}, \ldots, \xi_{m})$. Then there is a constant 
$c\!>\!0$ such that for any $m\times m$ matrix $T$ and $\eps\!>\!0$ 
we have $\mathrm{Prob} \left( \| T \xi \|_{2} \leq \varepsilon \| T\|_{HS}
\right) \leq \left( c c_0 \varepsilon\right)^{c \frac{ \| T\|_{HS}^{2}}
{\|T\|_{op}^{2}}}$. \qed 
\end{lem}
 
Our main lower bound for the condition number is then the following: 
\begin{lem}
\label{lower-bound}
\noindent Let $ P=(p_1, \ldots , p_m)$ be a homogeneous $n$-variate polynomial 
system with $d_j\!=\!\mathrm{deg}\; p_j$ for all $j$. Then $\tilde{\kappa}(P) 
\geq \frac{\|P\|_W}{\|P \|_{\infty}\sqrt{m+1}}$. 
Moreover, if $P\!:=\!(p_1, \ldots , p_m)$ is a random polynomial system 
satisfying our Sub-Gaussian and Euclidean Small Ball assumptions,  
with respective underlying constants $K$ and $\tilde{c_0}$, then we have 
$$ \mathrm{Prob} \left ( \tilde{\kappa}(P) \leq \varepsilon 
\frac{\sqrt{N}}{K md \log(ed)}  \right) \leq (c \tilde{c_0}
\varepsilon)^{ c^{\prime}
\min\left\{ N \frac{ \min_{j} N_{j}}{ \max_j N_{j}} , m 
d\log(ed)\right\} } \ \ \ \ {\rm and} $$
$$ \mathrm{Prob} \left ( \tilde{\kappa}(P) \leq \varepsilon \frac{\sqrt{N}}
{Km \log(ed)}  \right) \leq (c \tilde{c_0}\varepsilon)^{ c^{\prime} m 
\log(ed) } , \text{ if } d_j\!=\!d \text{ for all } j\!\in\!\{1,\ldots,
m\},$$
where $ c, c^{\prime}>0$ are absolute constants. In particular when $d=d_j$ 
for all $j\!\in\!\{1,\ldots,m\}$, we have 
${\mathbb{E}( \tilde{\kappa} (P) )} \geq c \frac{\sqrt{N}}{ m \log(ed)}$.   
\end{lem}

\begin{proof}
\noindent First note that Theorem \ref{kellog} implies that for every 
$ x, y \in S^{n-1}$ we have  
$$ \| d_j^{-1} D^{(1)} p_{j} (x) y\|_{2}^{2} \leq  \| p_{j}\|_{\infty}^{2}. $$
So we have $\| M^{-1} D^{(1)} P (x) (y) \|_{2}^{2} \leq \sum_{j=1}^{m} 
\|p_{j}\|_{\infty}^{2} \leq m \| P \|_{\infty}^{2}$. Now  
recall that\\ 
\mbox{}\hfill $\mathcal{L}^{2} (x,y) := \| M^{-1} D^{(1)} P (x) (y) \|_{2}^{2}  + \| p (x) \|_{2}^{2}$.\hfill\mbox{}\\ 
So we get $L^{2}(P):= \min\limits_{x\perp y}  
\mathcal{L}^{2}(x, y) \leq ( m +1) 
\| P \|_{\infty}^{2}$, which in turn implies that 
 $$ \tilde{\kappa}(P) \geq \frac{ \|P\|_W}{ L(P)} \geq \frac{\|P\|_W}{ 
\|P \|_{\infty}\sqrt{m+1}}  . $$
The proof for the case where $d_j=d$ for all $j \!\in\!\{1,\ldots,m\}$ is 
identical. 
 
We now show that, under our Euclidean Small Ball Assumption, we have that
\linebreak  
\scalebox{.95}[1]{$\mathrm{Prob} \left( \| P\|_W \leq \varepsilon 
\sqrt{N}\right) \leq \left( c 
\tilde{c_0 } \varepsilon\right)^{ c N \frac{ \min_j N_{j}}{\max_j N_{j}}}$  
for every $\varepsilon\!\in\!(0,1)$. Indeed, recall that 
\mbox{$\|p_{j} \|_W = \| C_{j} \|_{\ell_{2}^{N_{j}}}$}.}\linebreak 
\scalebox{.97}[1]{Then $\mathrm{Prob}\left( \|p_{j} \|_W \leq \varepsilon 
\sqrt{N_{j}} \right) 
\leq \left( \tilde{c_0} \varepsilon\right)^{N_{j}} \leq  \left( \tilde{c_0} 
\varepsilon\right)^{N_{j_0}}$ for any fixed $\varepsilon\!\in\!(0,1)$,   
where \mbox{$j_0\!\in\!\{1,\ldots,m\}$}}\linebreak  
satisfies $N_{j_0}:= \min_j N_{j}$. 
Let $\xi_{j}:= \frac{ \|p_{j}\|_W}{ \sqrt{N_{j}}}$ for any $j\!\in\!\{1,\ldots,
m\}$. Set $ \xi:= (\xi_{1}, \cdots , \xi_{m})$ \linebreak 
\scalebox{.95}[1]{and  
$T:= {\rm diag}( \sqrt{N_{1}}, \cdots , \sqrt{N_{m}})$. Note that 
$\|P\|_W = \| T \xi\|_{2}$, $ \|T\|_{HS} = \sqrt{\sum_{j=1}^{m} N_{j}}=
\sqrt{N}$, and}\linebreak 
\scalebox{.95}[1]{$\|T\|_{op} := \max_{1\leq j \leq m} \sqrt{N_{j}}$. Then 
Lemma \ref{Aniso-p} implies 
$\mathrm{Prob} \left( \|P\|_W \leq \varepsilon \sqrt{N}\right) \leq 
\left( c \tilde{c_0} \varepsilon\right)^{ c N \frac{ \min_j 
N_j}{ \max_j N_j}}$.}\linebreak  
Recall that Lemma  \ref{operatornorm} implies that for every $t\geq 1$ we 
have   
$$ \mathrm{Prob} \left( \| p\|_{\infty} \geq c t   K \sqrt{m} \log(ed)  
\right) \leq e^{-t^{2} m\log(ed)}.$$
So using our lower bound estimate for the condition number, we get 
$$ \mathrm{Prob} \left( \frac{ \|P\|_W}{ \| P \|_{\infty}} \geq 
\frac{ c^{\prime} \varepsilon \sqrt{N}}{tK \sqrt{m} \log(ed)} \right) \leq 
\mathrm{Prob} \left( \tilde{\kappa} (P) \geq \frac{ c \varepsilon \sqrt{N}}
{t Km d \log(ed)} \right), $$
$$ \mathrm{Prob} \left( \{ \|P\|_W \geq c^{\prime}\varepsilon\sqrt{N}\} \cap 
\{ \|P\|_{\infty} \leq ct K\sqrt{m} \log(ed)\} \right) \leq \mathrm{Prob} 
\left( \tilde{\kappa} (P) \geq \frac{ c \varepsilon \sqrt{N}}{t Km d \log(ed)} 
\right),$$ 

\noindent and

$$ \mathrm{Prob} \left( \{ \|P\|_W \geq c^{\prime}\varepsilon\sqrt{N}\} \cap 
\{ \|P\|_{\infty} \leq ct K\sqrt{m} \log(ed)\} \right) \geq  1 - \left( c \tilde{c_0} \varepsilon\right)^{ c N \frac{ \min_j 
N_j}{ \max_j N_j}} - e^{-t^{2} m\log(ed)} $$ 

\noindent We may choose $ t:= \sqrt{\log{\frac{1}{\varepsilon}}}$ and, by adjusting  
constants, we get our result. The case where $ d_j=d$ for all 
$j\!\in\!\{1,\ldots,m\}$ is similar. The bounds for the expectation follow by 
integration. \end{proof}

Observe that the dominant factor in the very last estimate of Lemma 
\ref{lower-bound} is $ \sqrt{N}$, which is the normalization coming 
from the Weyl-Bombieri norm of the polynomial system. So it makes sense to 
seek the asymptotic behavior of $ \frac{{\tilde{\kappa}}(P)}{ \sqrt{N}}$. When 
$m=n-1$, the upper bounds we get are exponential with respect to $n$, while 
the lower bounds are not. But when $m=2n-3$ and $ d= d_j$ for all $j\!\in\!\{1,\ldots,m\}$, we have the following upper bound (by Theorem \ref{expectation}) 
and lower bound (by Theorem \ref{lower-bound}): 
$$\frac{A_1}{ n d \log(ed)}  \leq \frac{ \mathbb{E} ( \tilde{\kappa}(P))  }
{\sqrt{N}} \leq  \frac{A_2\log{ed} \max \{ d^8 , n \} }{ \sqrt{n}},$$ 
where $A_1, A_2$ are constants depending on $(K,c_0)$. This suggests that 
our estimates are closer to optimality when $m$ is a constant multiple 
of $n$. 

\smallskip

\begin{rem}
There are similarities between our probability tail estimates 
and the older estimates in the linear case studied in \cite{RV-2}. In 
particular our estimates in the quadratic case $ d=2$, when $m$ is a 
constant multiple of $n$, are quite similar 
to the optimal result (for the linear case) appearing in \cite{RV-2}. 
\dia 
\end{rem}

\section{Acknowledgements}
We would like to thank the anonymous referee for detailed remarks, especially for the comment that helped us to spot a mistake in the previous proof of Theorem \ref{L-theorem}.

\end{document}